\documentclass[12pt]{amsart}
\usepackage{a4wide}
\usepackage{amsmath,amssymb,amsthm}
\usepackage{amsfonts,amsthm,amsmath,amssymb,amscd,color}
\usepackage{graphicx}
\usepackage{ulem}
\usepackage{enumerate}
\usepackage{parskip}

\usepackage{verbatim} 

\usepackage{amsmath,amssymb,amsthm}
\usepackage{enumerate}
\usepackage{parskip}
\usepackage{relsize}
\usepackage{graphicx}
\usepackage{color} 
\usepackage{tikz}  
\usepackage[iso]{isodateo}
\usepackage{bbm}
\usepackage[framemethod=TikZ]{mdframed}
\usepackage[many]{tcolorbox}
\newtcbtheorem[number within=section]{mytheo}{}%
{colback=white,colframe=magenta!50,fonttitle=\bfseries,
	enhanced,
	coltitle=black!75!black,
	attach boxed title to top left=
	{xshift=2ex,yshift=-2mm,yshifttext=-1mm},
	boxed title style={colframe=blue!50,
		colback=yellow!50}}{th}

\usepackage[toc,page]{appendix}
\newcommand{\RR}{\mathbb{R}}

\newcommand{\NN}{\mathbb{N}}

\newcommand{\TT}{\mathcal{T}}
\newcommand{\MM}{\mathcal{M}}

\newcommand{\PP}{P^1_{\frac{t}{m}}}
\newcommand{\PQ}{P^2_{\frac{t}{m}}}

\newcommand{\eps}{\varepsilon}

\newcommand{\CP}{\mathcal{P}}
\newcommand{\CM}{{\mathcal{M}}}

\newcommand{\BL}{\mathrm{BL}}

\newtheorem{thrm}{Theorem}[section]
\newtheorem{prop}[thrm]{Proposition}

\newtheorem{lemma}[thrm]{Lemma}
\newtheorem{clry}[thrm]{Corollary}

\newtheorem{remark}[thrm]{Remark}
\newtheorem{assumption}{Assumption}
\newtheorem{assumptionkw}{Assumption KW}
\newtheorem{assumptionCC2}{Assumption CG}

\newcommand{\smfrac}[2]{\mbox{$\frac{#1}{#2}$}}

\makeatletter
\@addtoreset{equation}{section}

\newcommand{\bes}{\begin{displaymath}}
\newcommand{\ees}{\end{displaymath}}
\newcommand{\be}{\begin{equation}}
\newcommand{\ee}{\end{equation}}
\newcommand{\ba}{\begin{eqnarray}}
\newcommand{\ea}{\end{eqnarray}}
\newcommand{\bas}{\begin{eqnarray*}}
	\newcommand{\eas}{\end{eqnarray*}}
\newcommand{\@Bbb}[1]{\ensuremath{\Bbb #1}}

\newcommand{\B}{{\@Bbb B}}
\newcommand{\C}{{\@Bbb C}}
\newcommand{\E}{{\@Bbb E}}
\newcommand{\F}{{\@Bbb F}}
\newcommand{\G}{{\@Bbb G}}
\renewcommand{\P}{{\@Bbb P}}

\newcommand{\Q}{{\@Bbb Q}}
\newcommand{\bQ}{{\@Bbb Q}}
\newcommand{\N}{{\@Bbb N}}
\newcommand{\R}{{\@Bbb R}}
\newcommand{\T}{{\@Bbb T}}
\newcommand{\bbR}{{\@Bbb R}}
\newcommand{\W}{{\@Bbb W}}
\newcommand{\Z}{{\@Bbb Z}}
\newcommand{\bbZ}{{\@Bbb Z}}

\newcommand{\@s}[1]{\ensuremath{\mathcal #1}}
\newcommand{\cA}{\@s A}
\newcommand{\cB}{\@s B}
\newcommand{\cC}{\@s C}
\newcommand{\cD}{\@s D}
\newcommand{\cE}{\@s E}
\newcommand{\cF}{\@s F}
\newcommand{\cG}{\@s G}
\newcommand{\cH}{\@s H}
\newcommand{\cI}{\@s I}
\newcommand{\cJ}{\@s J}

\newcommand{\cK}{\@s K}
\newcommand{\cL}{\@s L}
\newcommand{\cN}{\@s N}
\newcommand{\cM}{\@s M}
\newcommand{\cO}{\@s O}
\newcommand{\cP}{\@s P}
\newcommand{\cQ}{\@s Q}
\newcommand{\cR}{\@s R}
\newcommand{\cS}{\@s S}
\newcommand{\cT}{\@s T}
\newcommand{\cU}{\@s U}
\newcommand{\cV}{\@s V}
\newcommand{\cW}{\@s W}
\newcommand{\cX}{\@s X}
\newcommand{\cY}{\@s Y}
\newcommand{\cZ}{\@s Z}

\def\qed{\mbox{$\square$}}
\newcommand{\@bm}[1]{\ensuremath{\mathbf #1}}
\newcommand{\bma}{\@bm a}\newcommand{\bmA}{\@bm A}
\newcommand{\bmb}{\@bm b}\newcommand{\bmB}{\@bm B}
\newcommand{\bmc}{\@bm c}\newcommand{\bmC}{\@bm C}
\newcommand{\bmd}{\@bm d}\newcommand{\bmD}{\@bm D}
\newcommand{\bme}{\@bm e}
\newcommand{\bmf}{\@bm f}\newcommand{\bmF}{\@bm F}
\newcommand{\bmg}{\@bm g}\newcommand{\bmG}{\@bm G}
\newcommand{\bmh}{\@bm h}\newcommand{\bmH}{\@bm H}
\newcommand{\bmi}{\@bm i}\newcommand{\bmI}{\@bm I}
\newcommand{\bmj}{\@bm j}
\newcommand{\bmk}{\@bm k}\newcommand{\bmK}{\@bm K}
\newcommand{\bml}{\@bm l}
\newcommand{\bmm}{\@bm m}\newcommand{\bmM}{\@bm M}
\newcommand{\bmn}{\@bm n}
\newcommand{\bmo}{\@bm o}
\newcommand{\bmp}{\@bm p}
\newcommand{\bmq}{\@bm q}\newcommand{\bmQ}{\@bm Q}
\newcommand{\bmr}{\@bm r}
\newcommand{\bms}{\@bm s}\newcommand{\bmS}{\@bm S}
\newcommand{\bmt}{\@bm t}
\newcommand{\bmu}{\@bm u}\newcommand{\bmU}{\@bm U}
\newcommand{\bmw}{\@bm w}\newcommand{\bmW}{\@bm W}
\newcommand{\bmv}{\@bm v}\newcommand{\bmV}{\@bm V}
\newcommand{\bmx}{\@bm x}\newcommand{\bmX}{\@bm X}\newcommand{\bx}{\@bm x}
\newcommand{\bmy}{\@bm y}\newcommand{\bmY}{\@bm Y}\newcommand{\by}{\@bm y}
\newcommand{\bmz}{\@bm z}\newcommand{\bmZ}{\@bm Z}

\newcommand{\bmzero}{\@bm 0}

\newcommand{\@g}[1]{\ensuremath{\mathfrak #1}}
\newcommand{\gA}{\@g A}
\newcommand{\gD}{\@g D}
\newcommand{\gJ}{\@g J}
\newcommand{\gF}{\@g F}
\newcommand{\gM}{\@g M}
\newcommand{\gR}{\@g R}

\newcommand{\Lip}{\mathop{\mathrm{Lip}}}

\newcommand{\commentout}[1]{{}}

\newcommand{\CF}{\mathcal{F}}

\makeatother
\begin{document}
	
	\title[Lie-Trotter product formula for Markov semigroups]{Lie-Trotter product formula for locally equicontinuous and tight Markov semigroups}
	\author{Sander C. Hille}
	\address{Mathematical Institute, Leiden University, P.O. Box 9512, 2300 RA Leiden, The Netherlands, (SH, MZ)}
	\email{\{shille,m.a.ziemlanska\}@math.leidenuniv.nl}
	\author{Maria A. Ziemla\'nska}

	\subjclass[2000]{37A30, 47D07, 47N40, 37M25} 
	\keywords{Lie-Trotter product formula, Markov semigroups, commutator conditions}
	\thanks{The work of MZ has been partially supported by a Huygens Fellowship of Leiden University.}
	\maketitle

	\begin{abstract}
	In this paper we prove a Lie-Trotter product formula for Markov semigroups in spaces of measures. We relate our results to "classical" results for strongly continuous linear semigroups on Banach spaces or Lipschitz semigroups in metric spaces and show that our approach is an extension of existing results. As Markov semigroups on measures are usually neither strongly continuous nor bounded linear operators for the relevant norms, we prove the convergence of the Lie-Trotter product formula assuming that the semigroups are locally equicontinuous and tight. A crucial tool we use in the proof is a Schur-like property for spaces of measures. 
	\end{abstract}
	
	\section{Introduction}

The main purpose of this paper is to generalize the Lie-Trotter product formula for strongly continuous linear semigroups in a Banach space to Markov semigroups on spaces of measures. The Lie-Trotter formula asserts the existence and properties of the limit
\[\lim_{n\to \infty} \left[S^1_{\frac{t}{n}}S^2_{\frac{t}{n}}\right]^nx=:S_tx,\]
where $(S^1_t)_{t\geq 0}$ and $(S^2_t)_{t\geq 0}$ are strongly continuous semigroups of bounded linear operators. It may equally be considered as a statement considering the convergence of a \textit{switching scheme}. The key challenge is to overcome the difficulties that result from the observation that 'typically' Markov semigroups do not consist of bounded linear operators (in a suitable norm on the signed measures) nor need to be strongly continuous. Therefore, the available results do not apply. 

The Lie-Trotter product formula was originated by Trotter \cite{Trotter} in 1959 for strongly continuous semigroups, for which the closure of the sum of two generators was a generator of a semigroup given by the limit of the Lie-Trotter scheme, and generalized i.a. by Chernoff \cite{chernoff1974product} in 1974. This approach seems to be not general enough to be applicable in various numerical schemes however. As shown by Kurtz and Pierre in \cite{kurtz1980counterexample}, even if the sum of two generators is again a generator of strongly continuous semigroup, this semigroup may not be given by the limit of Lie-Trotter product formula as it may not converge. Consequently, the analysis of generators of semigroups can lead to non-convergent numerical splitting schemes. Hence, a different approach is needed. The analysis of commutator type conditions as in \cite{Kuhnemund_Wacker, Colombo2004} avoids considering generators and their domains and may be easier to verify. 

Splitting schemes were applied and played a very important role in numerical analysis and recently in the theory of stochastic differential equations to construct solutions of differential equations, e.g. work by Cox and  Van Neerven \cite{Cox_phd}. It was shown by Carrillo, Gwiazda and Ulikowska in \cite{Ulikowska_Carillo} that properties of complicated models, like structured population models, can be obtained by splitting the original model into simpler ones and analyzing them separately, which also leads to switching schemes of Lie-Trotter form. B\'atkai, Csom\'os and Farkas investigated Lie-Trotter product formulae for abstract nonlinear evolution equations with delay in \cite{Farkas_Batkai2017}, a general product formula for the solution of nonautonomous
abstract delay equations in \cite{Farkas_Batkai_2012} and analyzed the convergence of operator
splitting procedures in \cite{BATKAI2013315}. 

Our starting point is the conditions for convergence of the Lie-Trotter product formula formulated by K$\ddot{\mathrm{u}}$hnemund and Wacker in \cite{Kuhnemund_Wacker}. This result appears to be a very useful tool in proving the convergence of the Lie-Trotter scheme without the need to have knowledge about generators of the semigroups involved. However, the semigroups considered by K\"uhnemund and Wacker are assumed to be strongly continuous.
We extend K$\ddot{\mathrm{u}}$hnemund and Wacker's case to semigroups of Markov operators on spaces of measures and present weaker sufficient conditions for convergence of the switching scheme.  Our method of proof builds on \cite{Kuhnemund_Wacker}, while the specific commutator condition that we employ (Assumption \ref{Assumption_commutator}) is motivated by \cite{Colombo2004}.

The theory of Markov operators and Markov semigroups was studied by Lasota, Mackey, Myjak and Szarek in the context of fractal theory \cite{SzarekMyjak,Lasota_Mackey}, iterated function systems and stochastic differential equations \cite{LASOTA2006513}. Markov semigroups acting on spaces of (separable) measures are usually not strongly continuous. The local equicontinuity (in measures) and tightness assumptions we employ are less restrictive and follow from strong continuity. The concept of equicontinuous families of Markov operators can be found in e.g. Meyn and Tweedie \cite{Meyn:2009:MCS:1550713}. Also, Worm in \cite{Worm} extends results of Szarek to families of equicontinuous Markov operators. 

The outline of the paper is as follows: in Section \ref{section_main_theorems} we present main results of this paper. Theorem \ref{main} in Section \ref{section_main_theorems} is the convergence theorem and is the most important result in the paper. The other important and non-trivial result is Theorem \ref{Theorem_composition}. Section \ref{section preliminaries} introduces Markov operators and Markov-Feller semigroups on the space of signed Borel measures $\MM(S)$, investigates their topological properties and consequences of equicontinuity and tightness of family of Markov operators. 
In Section \ref{section composition} we give tools to prove Theorem \ref{Theorem_composition}, i.e. that a composition of equicontinuous and tight families of Markov operators is again an equicontinuous and tight family. This result is quite delicate and seems like it was not considered in the literature before. We also provide a proof of the observation in Lemma \ref{precompact} which says that a family of equicontinuous and tight family of Markov operators on a precompact subset of positive measures is again precompact.  The proof of Theorem \ref{Theorem_composition} can be found in Appendix \ref{proof_theorem_composition}.\\
In Section \ref{Lie-Trotter} we  prove the convergence of the Lie-Trotter product formula for Markov operators. We provide more general assumptions then those provided in the K\"uhnemund-Wacker paper (see \cite{Kuhnemund_Wacker}). As our semigroups are not strongly continuous and usually not bounded, we use the concept of (local) equicontinuity (see e.g. Chapter 7 in \cite{Worm}). This allows us to define a new admissible metric $d_\mathcal{E}$ and a new $\|\cdot\|_{\BL,d_\mathcal{E}}$-norm dependent on the operators and the original metric $d$ on $S$. The crucial assumption is the Commutator Condition Assumption \ref{Assumption_commutator}.\\ To prove convergence of our scheme under Assumptions \ref{Assumption_equicontinuity_tightness}-\ref{Assumption_4_extended_commutator} we use a Schur-like property for signed measures, see \cite{Schur_like}, which allows us to prove weak convergence of the formula and conclude the strong/norm convergence.  In  Section \ref{subsection technical lemmas} we show crucial technical lemmas. The proofs of most lemmas from Section \ref{subsection technical lemmas} can be found in the Appendices \ref{App_2_Proof_of_Lemma_lem_est}-\ref{App_3_Proof_of_Lemma_lem_2 new}. In Section \ref{subsection technical lemmas} several useful properties of the limit operators that result from the converging Lie-Trotter formula are derived. \\
Section \ref{section relations to literature} shows that our approach is a generalization of K\"uhnemund-Wacker \cite{Kuhnemund1} and Colombo-Corli \cite{Colombo2004} cases. We show that if we consider Markov semigoups coming from lifts of deterministic operators, then the K\"uhnemund-Wacker and Colombo-Corli assumptions imply our assumptions and their convergence results of the Lie-Trotter formula or switching scheme follows from our main convergence result. 
	\section{Main theorems}
		\label{section_main_theorems}
Let $S$ be a Polish space, i.e. a separable completely metrizable topological space, see \cite{Worm}. Any metric $d$ that metrizes the topology of $S$ such that $(S,d)$ is separable and complete is called \textit{admissible}.  Let $d$ be an admissible metric on $S$. Following \cite{Dudley1966}, we denote the vector space of all real-valued Lipschitz functions on $(S,d)$ by $\Lip(S,d)$. For $f\in\Lip(S,d)$ we denote the Lipschitz constant of $f$ by
\begin{align}
\nonumber
\label{lipschitz_constant}
|f|_{L,d}:=\sup\left\{ \frac{|f(x)-f(y)|}{d(x,y)}:x,y\in S, x\not =y \right\} \end{align}  $\BL(S,d)$ is the subspace of bounded functions in $\Lip(S,d)$. Equipped with the bounded Lipschitz norm 
\[\|f\|_{\BL,d}:=\|f\|_\infty+|f|_{L,d}\] it is a Banach space, see \cite{Dudley1966}.
The vector space of finite signed Borel measures on $S$, $\MM(S)$, embeds into the dual of $(\BL(S), \|\cdot\|_{\BL,d})$, see \cite{Dudley1966}, thus introducing the dual bounded Lipschitz norm
$\|\cdot\|_{\BL,d}^*$ on $\MM(S)$ \begin{equation}\label{BL_norm}\|\mu\|_{\BL,d}^*:=\sup\left\{|\langle \mu,f\rangle|:f\in \mathrm{\BL}(S,d),\|f\|_{\BL,d}=\|f\|_\infty+|f|_{L,d}\leq 1\right\},\end{equation} for which the space becomes a normed space. It is not complete unless $(S,d)$ is uniformly discrete (see \cite{Worm}, Corollary 2.3.14). The cone $\MM^+(S)$ of positive measures in $\MM(S)$ is closed \cite{Worm,Dudley1966}. $\mathcal{P}(S)$ is the convex subset of $\MM^+(S)$ of probability measures.  The topology on $\MM(S)$ induced by $\|\cdot\|_{\BL,d}^*$ is weaker then the norm topology associated with the total variation norm $\|\mu\|_{TV}:=\mu^+(S)+\mu^-(S)$, where $\mu=\mu^+-\mu^-$ is the Jordan decomposition of $\mu$ (see Bogachev I, \cite{bogachev_1}, p.176). \\
 We define a Markov operator on $S$ to be a map $P:\MM^+(S)\to\MM^+(S)$ such that
\begin{itemize}
	\item [(i)] $P$ is additive and $\RR_+$-homogeneous;
	\item [(ii)] $\|P\mu\|_{TV}=\|\mu\|_{TV}$ for all $\mu\in\MM^+(S)$.
\end{itemize}

 Let $(P_\lambda)_{\lambda\in\Lambda}$ be a family of Markov operators.\\
Following Lasota and Szarek \cite{LASOTA2006513}, and Worm \cite{Worm}, we say that $(P_\lambda)_{\lambda\in\Lambda}$ is\textit{ equicontinuous} at $\mu\in\MM^+(S)$  if for every $\eps>0$ there exists $\delta>0$ such that $\left\|P_\lambda\mu-P_\lambda\nu\right\|^*_{\BL,d}<\eps$ for every $\nu\in\MM^+(S)$ such that $\|\mu-\nu\|_{\BL,d}^*<\delta$ and for every $\lambda\in\Lambda$.
$(P_\lambda)_{\lambda\in\Lambda}$ is called equicontinuous if it is equicontinuous at every $\mu\in\MM^+(S)$.
We will examine properties of space of bounded Lipschitz functions is Section \ref{section preliminaries}.

Let $\Theta\subset\mathcal{P}(S)$. Following \cite{bogachev_2} we call $\Theta$ \textit{uniformly tight} if for every $\epsilon>0$ there exists a compact set $K_\epsilon\subset S$ such that $\mu(K_\epsilon)\geq 1-\epsilon$ for all $\mu\in \Theta$.

The following theorem is a crucial tool for proving convergence of Lie-Trotter scheme for Markov semigroups but also an important and non-trivial result on its own. Proof of Theorem \ref{Theorem_composition} can be found in Section \ref{proof_theorem_composition}. 
	\begin{thrm}
		\label{Theorem_composition}
		Let $(P_\lambda)_{\lambda\in\Lambda}$, $(Q_\gamma)_{\gamma\in\Gamma}$	be equicontinuous families of Markov operators on $(S,d)$. Assume that $(Q_\gamma)_{\gamma\in\Gamma}$  is tight. Then the family $\left\{P_\lambda Q_\gamma:\lambda\in\Lambda, \gamma\in\Gamma\right\}$ is equicontinuous on $(S,d)$. Moreover, if $(P_\lambda)_{\lambda\in\Lambda}$ is tight, then the family $\left\{P_\lambda Q_\gamma:\lambda\in\Lambda, \gamma\in\Gamma\right\}$ is tight on $(S,d)$. 
	\end{thrm}
We now present assumptions under which we prove convergence of Lie-Trotter scheme. Even though they may seem technical, they are motivated by existing examples of convergence of Lie-Trotter schemes with weaker assumptions then those in \cite{Kuhnemund_Wacker,Colombo2004} (see Section \ref{section relations to literature}).

Let $(P^1_t)_{t\geq 0}$ and $(P^2_t)_{t\geq 0}$ be Markov semigroups.	Let $\delta>0$. Define \[\mathcal{P}^i(\delta):=\{P^i_t:t\in [0,\delta]\} \text{ for } i=1,2,\]  
\[\mathcal{F}(\delta):=\left\{\left[P^1_{\frac{t}{n}}P^2_{\frac{t}{n}}\right]^n:n\in\NN, t\in[0,\delta]\right\}.\]  Let $d$ be an admissible metric on $S$ such that the following assumptions hold:

	\begin{assumption} 
		\label{Assumption_equicontinuity_tightness}There exists $\delta_1>0$ such that
	  $\mathcal{P}^1(\delta_1)$ and $\mathcal{P}^2(\delta_1)$ are equicontinuous and tight families of Markov operators on $(S,d)$.
	\end{assumption}

	\begin{assumption}  [Stability condition] \label{Assumption_stability} There exists $\delta_2>0$ such that  $\mathcal{F}(\delta_2)$ is an equicontinuous family of Markov operators on $(S,d)$. 
	\end{assumption}
Under Assumption \ref{Assumption_equicontinuity_tightness}, the operators $P_t^i, 0\leq t\leq\delta$, are Feller: there exist ${U_t^i:C_b(S)\to C_b(S)}$ such that $\langle P_t^i\mu,f\rangle=\langle \mu,U_t^if\rangle$ for every $f\in C_n(S),\mu_0\in\MM^+(S)$, $0\leq t\leq \delta$.

Let $f\in \BL(S,d)$ and consider
\begin{align}
\label{def_E}
\mathcal{E}(f):=\left\{U_s^2U_{s'}^1\left[U^2_{\frac{t}{n}}U^1_{\frac{t}{n}}\right]^nf:n\in\NN, s,s',t\in[0,\delta]\right\}
\end{align}
  By Theorem 7.2.2 in  \cite{Worm} or Theorem \ref{Thrm_equivalence} below, equicontinuity of the family $(P_\lambda)_{\lambda\in\Lambda}$ is equivalent to equicontinuity of the family $(U_\lambda f)_{\lambda\in\Lambda}$ for every $f\in \BL(S,d)$. Then, as we will show in Lemma \ref{lemma_equicontinuity_of_E}, $\mathcal{E}(f)$ is an equicontinuous family if $\delta\leq \min(\delta_1,\delta_2)$. It defines a new admissible metric on $S$:
\begin{equation}
\label{new_metric}
d_{\mathcal{E}(f)}(x,y):=d(x,y)\vee \sup_{g\in\mathcal{E}(f)}|g(x)-g(y)|,\quad \text{for}\quad x,y\in S.
\end{equation}

\begin{assumption}	[Commutator condition]
	\label {Assumption_commutator} There exists a dense convex  subcone $M_0$ of $\MM^+(S)_{\BL,d}$ that is invariant under $(P^i_t)_{t\geq 0}$ for $i=1,2$ and for every $f\in \BL(S,d)$ there exists $\delta_{3,f}>0$ such that for the admissible metric $d_{\mathcal{E}(f)}$ on $S$ there exists ${\omega_f:[0,\delta_{3,f}]\times M_0\to \RR_+}$ continuous, non-decreasing in the first variable, such that the Dini-type condition holds \begin{equation}\label{equation_commutator_condition}
	\int_0^{\delta_{3,f}} \frac{\omega_f(s,\mu_0)}{s}ds<+\infty\quad \text{for all }\quad \mu_0\in M_0, \text{ and }
	\end{equation}
	\[\left\|P^1_tP^2_t\mu_0-P^2_tP^1_t\mu_0\right\|^*_{\BL,d_{\mathcal{E}(f)}}\leq t\omega_f(t,\mu_0) \] for every $t\in [0,\delta_{3,f}], \mu_0\in M_0$. 
\end{assumption}
\begin{assumption}[Extended Commutator Condition] 
	\label{Assumption_4_extended_commutator}
	Assume that Assumption \ref{Assumption_commutator} holds and, in addition, for every $f\in \BL(S,d)$, there exists $\delta_{4,f}>0$ and for  $\mu_0\in M_0$ there exists $C_f(\mu_0)>0$ such that for every $t\in [0,\delta_{4,f}]$,
	\[\omega_f(t,P\mu_0)\leq C_f(\mu_0)\omega_f(t,\mu_0)\] for all $P\in\mathcal{P}^2(\delta_{4,f})\cdot \mathcal{F}(\delta_{4,f})\cdot\mathcal{P}^1(\delta_{4,f})$.
\end{assumption}
Now we can formulate the main theorem of this paper, which is the strong convergence of the Lie-Trotter scheme. The proof of Theorem \ref{main} can be found in Section \ref{proof_main}.
	\begin{thrm}
		\label{main}
		Let $(P^1_t)_{t\geq 0}$ and $(P^2_t)_{t\geq 0}$ be semigroups of  Markov operators. Assume that Assumptions \ref{Assumption_equicontinuity_tightness}-\ref{Assumption_4_extended_commutator} hold.  Then for every $t\geq 0$ there exists a unique Markov operator $\mathbb{\overline{P}}_t:\MM^+(S)\to \MM^+(S)$ such that for every $\mu\in\MM^+(S)$:
		\begin{equation}
		\label{eq_convergence}\left\|\left[P^1_{\frac{t}{n}}P^2_{\frac{t}{n}}\right]^n\mu-\mathbb{\overline{P}}_t\mu\right\|^*_{\BL,d}\to 0 \text{ as } n\to\infty\end{equation}  
		If, additionally, a single $\delta_{3,f}$, $\delta_{4,f}$, $C_f(\mu_0)$ and $\omega_f(\cdot,f)$ can be chosen to hold uniformly for $f\in\BL(S,d)$, $\|f\|_{\BL,d}\leq 1$, then convergence in (\ref{eq_convergence}) is uniform for $t$ in compact subsets of $\RR_+$.
	\end{thrm}

\section{Preliminaries} 
\label{section preliminaries}
\subsection{Markov operators and semigroups}
We start with some preliminary results on Markov operators on spaces of measures, see \cite{Worm,ethier,LasotaMyjak}. Let $S$ be a Polish space,
 $P:\MM^+(S)\to\MM^+(S)$ a Markov operator.
 We extend $P$ to a positive bounded linear operator on $(\MM(S),\|\cdot\|_{TV})$ by $P\mu:=P\mu^+-P\mu^-$. $P$ is a bounded linear operatos on $\MM(S)$ for $\|\cdot\|_{TV}$. 'Typically' it is not bounded for $\|\cdot\|_{\BL,d}^*$. Denote by $\mathrm{BM}(S)$ the space of all bounded Borel measurable functions on $S$. Following \cite{Hille2009}, Definition 3.2 or \cite{SzarekMyjak} we will call a Markov operator $P$ \textit{regular} if there exists $U:\mathrm{BM}(S)\to \mathrm{BM}(S)$ such that
 \[\langle P\mu,f\rangle=\langle \mu,Uf\rangle \text{ for all }\mu\in\MM^+(S), f\in \mathrm{BM}(S).\]
Let $(S,\Sigma)$ be a measurable space. According to \cite{Worm}, Proposition 3.3.3, $P$ is regular if and only if
\begin{itemize}
 		\item [(a)] $x\mapsto P\delta_x(E)$ is measurable for every $E\in\Sigma$ and
 		\item [(b)] $P\mu(E)=\int_{S}P\delta_x(E) d\mu(x)$ for all $E\in\Sigma$.
\end{itemize}
 We call the operator $U:\mathrm{BM}(S)\to \mathrm{BM}(S)$ the \textit{dual operator} of $P$.\\
The Markov operator $P$ is a \textit{Markov-Feller operator} if it is regular and the dual $U$ maps $C_b(S)$ into itself.  A Markov semigroup $(P_t)_{t\geq 0}$ on $S$ is a semigroup of Markov operators on $\MM^+(S)$. The Markov semigroup is regular (or Feller) if all the operators $P_t$ are regular (or Feller). Then $(U_t)_{t\geq 0}$ is a semigroup on $\mathrm{BM}(S)$, which we call the \textit{dual semigroup}.
 \subsection{Topological preliminaries}
Following \cite{kelley1955general}, p.230, a topological space $X$ is a \textit{k-space} if for any subset $A$ of $X$ holds that if $A$ intersects each closed compact set in a closed set, then $A$ is closed.
According to  \cite{engelking1977general}, Theorem 3.3.20 every first-countable Hausdorff space is a $k$-space. Every metric space is first countable, hence also a $k$-space. In particular $(\MM^+(S),\|\cdot\|_{\BL,d}^*)$ is a $k$-space.\\
Let $\mathcal{F}$ be a family of continuous maps from a topological space $X$ to a metric space $(Y,d_Y)$. $\mathcal{F}$ is \textit{equicontinuous at point} $x\in X$ if for every $\eps>0$ there exists an open neighbourhood $U_\eps$ of $X$ in $X$ such that
\[d_Y(f(x),f(x'))<\eps\text{ for all } x'\in U_\eps, \forall f\in \mathcal{F}.\] 
A family $\mathcal{F}$ of maps is \textit{equicontinuous} if and only if it is \textit{equicontinuous at every point}. A family $\mathcal{F}$ of maps from a metric space $(X,d_X)$ to a metric space $(Y,d_Y)$ is \textit{uniformly equicontinuous} if for every $\eps>0$ there exists $\delta_\eps>0$ such that
\[d_Y(f(x),f(x'))<\eps\text{ for all } x,x'\in X\text{ such that } d_X(x,x')<\delta_\eps\text { for all } f\in\mathcal{F}.\]
\begin{lemma}
	\label{uniform_equicontinuity}
	Let $(K,d)$ be a compact metric space and $(Y,d_Y)$ a metric space. An equicontinuous family $\mathcal{F}\subset \mathcal{C}(K,Y)$ is uniformly equicontinuous.
\end{lemma}
\begin{proof}
	Let $\eps>0$. For each $x\in K$ there exists an open ball $B_x(\delta_x)$, $\delta_x>0$ such that $d_Y(f(f),f(x'))<\eps$ for every $x'\in B_x(\delta_x)$ and $f\in\mathcal{F}$. By compactness of $K$, it is covered by finitely many balls, $B_{x_i}(\delta_{x_i}/2), i=1,\cdots,n$, say. Let $\delta:=\min_i\frac{\delta_{x_i}}{2}$. If $x,x'\in K$ are such that $d(x,x')<\delta$, then there exists $x_{i_0}$ such that $x\in B_{x_{i_0}}(\delta_{x_{i_0}}/2)$. Necessarily,
	\[d(x',x_{i_0})\leq d(x',x)+d(x,x_{i_0})<\delta+\delta_{x_{i_0}}/2<\delta_{x_{i_0}}.\] Thus, $d_Y(f(x),f(x'))<\eps$, proving the uniform equicontinuity on $K$. 
\end{proof}
For a family of maps $\mathcal{F}$ on $X$ and $x\in X$ we write $\mathcal{F}[x]:=\{f(x):f\in\mathcal{F}\}.$ Following \cite{kelley1955general} we introduce the compact-open topology. Let $X,Y$ be topological spaces. Let $F$ denote a non-empty set of functions from $X$ to $Y$. For each subset $K$ of $X$ and each subset $U$ of $Y$, define $W(K,U)$ to be the set of all members of $F$ which carry $K$ into $U$; that is $W(K,U):=\{f:f[K]\subset U\}$. The family of all sets of the form $W(K,U)$, for $K$ a compact subset of $X$ and $U$ open in $Y$, is a subbase for the compact-open topology for $F$. The family of finite intersections of sets of the form $W(K,U)$ is then a base for the compact open topology. We write co-topology as abbreviation for compact-open topology. For two topological spaces $T$ and $T$, $C(T,T')$ is the set of continuous maps from $T$ to $T'$.  The following generalized Arzela-Ascoli type theorem is based on  \cite{kelley1955general}, Theorem 7.18.
\begin{thrm}
	\label{aa_kelley}
 Let $\mathcal{C}$ be the family of all continuous maps from a $k$-space $X$ which is either Hausdorff or regular to a metric space $(Y,d)$, and let $\mathcal{C}$ have the co-topology. Then a subfamily $\mathcal{F}$ of $\mathcal{C}$ is compact if and only if:
	\begin{itemize}
		\item [(a)] $\mathcal{F}$ is closed in $\mathcal{C}$;
		\item [(b)] the closure of $\mathcal{F}[x]$ in $Y$ is compact for each $x$ in $X$;
		\item [(c)] $\mathcal{F}$ is equicontinuous on every compact subset of $X$.
	\end{itemize}
\end{thrm}
\begin{thrm}
	\label{aa_b_y}
	[Bargley and Young \cite{BargleyYang}, Theorem 4] Let $X$ be a Hausdorff $k$-space and $Y$ a Hausdorff uniform space. Let $\mathcal{F}\subset C(X,Y)$. Then $\mathcal{F}$ is compact in the co-topology if and only if
	\begin{itemize}
		\item [(a)] $\mathcal{F}$ is closed;
		\item [(b)] $\mathcal{F}[x]$ has compact closure for each $x\in X$;
		\item [(c)] $\mathcal{F}$ is equicontinuous.
	\end{itemize}
\end{thrm}
which is a generalization of Theorem 8.2.10 in \cite{engelking1977general}. This yields the conclusion that for a closed family of continuous functions $\mathcal{F}$ such that $\mathcal{F}[x]$ is precompact for every $x$, equicontinuity on compact sets is equivalent to continuity.\\ 
Moreover, Theorem \ref{aa_b_y} can be rephrased for a family $\CF$ that is relatively compact in $\mathcal{C}$, meaning that its (compact-open) closure is compact:

\begin{thrm}\label{thrm:char rel compactness}
	Let $X$ be a Hausdorff $k$-space and $Y$ a metric space. Let $\mathcal{C}= C(X,Y)$, equipped with the co-topology. A subset $\CF$ of $\mathcal{C}$ is relatively compact iff:
	\begin{enumerate}
		\item[(a)] The closure of $\CF[x]:=\{f(x): f\in\CF\}$ in $Y$ is compact for every $x\in X$.
		\item[(b)] $\CF$ is equicontinuous on every compact subset of $X$.
	\end{enumerate}
	Statement (b) can be replaced by
	\begin{enumerate}
		\item[(b')] $\CF$ is equicontinuous on $X$.
	\end{enumerate}
\end{thrm}

\begin{proof}
	Let $\overline{\CF}$ be the closure of $\CF$ in $\mathcal{C}$. Assume it is compact, then according to Theorem \ref{aa_kelley}, the closure of $\overline{\CF}[x]$ in $Y$ is compact for every $x\in X$. Hence the closure of $\CF[x]$, which is contained in the closure of $\overline{\CF}[x]$, will be compact too. The family $\CF$ is equicontinuous on $X$ for every compact subset of $X$, because it is a subset of $\overline{\CF}$ that has this property.\\
	On the other hand, if $\CF$ satisfies (a) and (b), or (b'), then $\overline{\CF}$ obviously satisfies condition (a) in Theorem \ref{aa_kelley}. Now let $f\in\overline{\CF}$. Then there exists a net $(f_\nu)\subset \CF$ such that $f_\nu\to f$. Point evaluation at $x$ is continuous for the co-topology, so $f_\nu(x)\to f(x)$ in $Y$. Since $f_\nu(x)$ is contained in a compact set in $Y$ for every $\nu$, $f(x)$ will be contained in this compact set too. So (b) holds in Theorem \ref{aa_kelley} for $\overline{\CF}$. In a similar way one can show (c) in Theorem \ref{aa_kelley}. Let $K\subset X$ be compact. The co-topology on $C(X,Y)$ is identical to the topology of uniform convergence on compact subsets (cf. \cite{kelley1955general}, Theorem 7.11). So if $f_*\in \overline{\CF}$ and $(f_\nu)\subset \CF$ is a net such that $f_\nu\to f_*$, then $f_\nu|_K\to f_*|_K$ uniformly. If $x_0\in K$, then for every $\eps>0$ there exists an open neighbourhood $U$ of $x_0$ in $K$ such that
	\[
	d_Y(f(x),f(x_0))<\smfrac{1}{2}\eps\qquad \mbox{for all}\ f\in\CF,\ x\in U.
	\]
	Consequently, 
	\[
	d_Y(f_*(x),f_*(x_0)) = \lim_\nu d_Y(f_\nu(x),f_\nu(x_0)) \leq \smfrac{1}{2}\eps <\eps
	\]
	for all $x\in U$. So $\overline{\CF}$ is equicontinuous on $K$ too. Theorem \ref{aa_kelley} then yields the compactness of $\overline{\CF}$ in $\mathcal{C}$, hence the relative compactness of $\CF$.
\end{proof}
In \cite{Worm} and in \cite{Schur_like} we can find the following result, which will be crucial in proving norm convergence of the Lie-Trotter product formula. 
\begin{thrm} 
	\label{weak_implies_strong}
	Let $S$ be complete and separable. Let $(\mu_n)_{n\in\NN}\subset\MM_s(S)$ and $N\geq 0$ be such that $\langle \mu_n,f\rangle$ converges as $n\to\infty$ for every $f\in \BL(S)\simeq \MM(S)^*_{\BL}$ and \[\|\mu_n\|_{TV}\leq N\quad for\,\, every\,\, n\in\NN.\]
	Then there exists $\mu\in\MM(S)$ such that $\|\mu_n-\mu\|_{\BL}^*\to 0$ as $n\to\infty$.
\end{thrm}
\subsection{Tight Markov operators} 
Let us now introduce the concept of tightness of sets of measures and families of Markov operators. \\
According to  \cite{bogachev_2}, Theorem 7.1, all Borel measures on a Polish space are Radon i.e.  locally finite and inner regular. Also, by Definition 8.6.1 in \cite{bogachev_2} we say that a family of Radon measures $\mathcal{M}$ on a topological space $S$ is called \textit{uniformly tight} if for every $\eps>0$, there exists a compact set $K_\eps$ such that $|\mu|(S\setminus K_\eps)<\eps$ for all $\mu\in\mathcal{M}$. Moreover, we say that a family $(P_\lambda)_{\lambda\in\Lambda}$ of Markov operators is \textit{tight} if for each $\mu\in\MM^+(S)_{\BL}$, $\{P_\lambda\mu:\lambda\in\Lambda\}$ is \textit{uniformly tight}.
The following theorem, which is a rephrased version of  Theorem 8.6.2 in \cite{bogachev_2}, due to Prokhorov shows that 
in our case tightness of the $\|\cdot\|_{TV}$-uniformly bounded family is equivalent to precompactness of $\{P_\lambda\mu\,|\,\lambda\in\Lambda\}$ in $\MM^+(S)_{\BL}$. 
\begin{thrm}[Prokhorov theorem]
 Let $S$ be a complete separable metric space and let $M$ be a family of finite Borel measures on $S$. The following conditions are equivalent:
	\begin{itemize}
		\item [(i)] Every sequence $\{\mu_n\}\subset M$ contains a weakly convergent subsequence.
		\item [(ii)] The family $M$ is uniformly tight and uniformly bounded in total variation norm.
	\end{itemize}
\end{thrm}
\section{Equicontinuous families of Markov operators}
\label{section composition}
Let $S$ be a Polish space and consider a semigroup $(P_t)_{t\geq 0}$ of Markov operators. We will examine properties of equicontinuous families of Markov operators. An equicontinuous family of Markov operators must consist of $\|\cdot\|_{\BL,d}^*$-continuous operators. These are Feller (cf. \cite{Worm}, Lemma 7.2.1).
Due to Theorem \ref{aa_kelley}, a closed subset $F$ of the mappings from $\MM^+(S)_{\BL}$ to $\MM^+(S)_{\BL}$ with the co-topology is compact if and only if $F|_K$ is equicontinuous for each compact $K\subset \MM^+(S)$ and the set $\{P_t\mu:P_t\in F\}\subset \MM^+(S)$ has a compact closure for every $\mu\in\MM^+(S)$. A continuous function on a compact metric space is uniformly continuous.  A similar statement holds for equicontinuous families.
\begin{lemma}\label{L1} Let $(P_\lambda)_{\lambda\in\Lambda}$ be a family of Markov operators on $S$. If $(P_\lambda)_{\lambda\in \Lambda}$ is an equicontinuous family on the compact set $K\subset\MM^+(S)$, then $(P_\lambda)_{\lambda\in\Lambda}$ is uniformly equicontinuous on $K$.
\end{lemma}
The following result, found in \cite{Schur_like} and based on \cite{Worm}, Theorem 7.2.2, gives equivalent conditions for a family of regular Markov operators to be equicontinuous:
\begin{thrm}
	\label{Thrm_equivalence}
Let $(P_\lambda)_{\lambda\in\Lambda}$ be a family of regular family of Markov operators on the complete separable metric space $(S,d)$. Let $U_\lambda$ be the dual operator of $P_\lambda$. Then the following statements are equivalent:
\begin{itemize}
	\item [(i) ] $(P_\lambda)_{\lambda\in\Lambda}$ is an equicontinuous family;
	\item [(ii) ] $(U_\lambda f)_{\lambda\in\Lambda}$ is an equicontinuous family in $C_b(S)$ for all $f\in \mathrm{\BL}(S,d)$;
	\item [(iii) ] $\{U_\lambda f|f\in B, \lambda\in\Lambda\}$ is an equicontinuous family for every bounded set $B\subset \mathrm{\BL}(S,d)$.
\end{itemize}
\end{thrm}
 In the next part of this section we show results which allow us to prove Theorem \ref{Theorem_composition}, that is that the composition of equicontinuous family of Markov operators with equicontinuous and tight family of Markov operators is equicontinuous. Additionally, if both families are tight, the composition is also tight. One can find an example of equicontinuous and tight families of Markov operators in \cite{Szarek_dissertationes}. 
 
 Let us first prove the following crucial observation. 
 \begin{lemma}
 	\label{precompact}
 	Let $(P_\lambda)_{\lambda\in\Lambda}$ be an equicontinuous and tight family of Markov operators on $(S,d)$ and let $K\subset \MM^+(S)_{\BL}$ be precompact. Then $\{P_\lambda\mu\,|\,\mu\in K,\lambda\in\Lambda\}\subset\MM^+(S)_{\BL}$ is precompact.
 \end{lemma}
 \begin{proof}
 	As $K$ is precompact, then $\overline{K}$ is compact in $\MM^+(S)_{\BL}$. So $(P_\lambda|_{\overline{K}})\subset C(\overline{K},\MM^+(S)_{\BL})$ is equicontinuous and for each $\mu\in\bar{K}$, $\{P_\lambda\mu|\lambda\in\Lambda\}$ is precompact, by tightness of the family $(P_\lambda)_{\lambda\in\Lambda}$. 
 	Hence, by Theorems \ref{aa_kelley}-\ref{aa_b_y}, $\{P_\lambda|_{\overline{K}}\}\subset C(\overline{K},\MM^+(S)_{\BL})$ is relatively compact for the compact-open topology, which is the $\|\cdot\|_{\infty}$-norm topology in this case.\\
 	Let us consider the evaluation map
 	\begin{equation}
 	\nonumber
 	\begin{array}{rcl}
 	ev:C(\overline{K},\MM^+(S)_{\BL})\times \overline{K}&\to&\MM^+(S)_{\BL}\\
 	(F,\mu)&\mapsto&F(\mu).
 	\end{array}
 	\end{equation}
 	 Theorem 5, \cite{kelley1955general},  p.223 yields that this map is jointly continuous if $C(\overline{K},\MM^+(S)_{\BL})$ is equipped with the co-topology.
 	So 
 	\[K'=\{F(\mu)\,|\,F\in \mathrm{Cl}(\{P_\lambda|_{\overline{K}}:\lambda\in\Lambda\}), \mu\in\overline{K}\}\] is compact in $\MM^+(S)_{\BL}$. 
 \end{proof}
To prove Theorem \ref{Theorem_composition}, we will need the following result.
\begin{prop}\label{prop:e-prop for one is e-prop for any}
	Let $(P_\lambda)_{\lambda\in\Lambda}$ be a tight family of regular Markov operator on $S$. If $(P_\lambda)_{\lambda\in\Lambda}$ is equicontinuous for one admissible metric on $S$, then it is equicontinuous for any admissible metric.
\end{prop}
The key point in the proof of Proposition \ref{prop:e-prop for one is e-prop for any} is a series of results on characterisation of compact sets in the space of continuous maps when equipped with the co-topology.
These can be stated in quite some generality, originating in \cite{kelley1955general, engelking1977general, BargleyYang}.
\begin{proof}  Let $d$ be the admissible metric on $S$ for which $(P_\lambda)$ is equicontinuous in $\mathcal{C}_d:=C(\CP(S)_{\mathrm{weak}}, \CP(S)_{\BL,d})$. Let $d'$ be any other admissible metric on $S$. We must show that $(P_\lambda)$ is an equicontinuous family in $\mathcal{C}_{d'}:=C(\CP(S)_{\mathrm{weak}}, \CP(S)_{\BL,d'})$.\\
	By assumption, $\{P_\lambda\mu:\lambda\in\Lambda\}$ is tight for every $\mu\in\CP(S)$. By Prokhorov's Theorem, it is relatively compact in $\CP(S)_{\BL,d}$, because the $\|\cdot\|_{\BL,d}$-norm topology coincides with the weak topology on $\CM^+(S)$. Because $(P_\lambda)$ is equicontinuous in $\mathcal{C}_d$, Theorem \ref{thrm:char rel compactness} yields that $(P_\lambda)$ is relatively compact in $\mathcal{C}_d$, for the co-topology. Since the topologies on $\CP(S)$ defined by the norms $\|\cdot\|_{\BL,d'}$, $d'$ admissible, all coincide with the weak topology, $(P_\lambda)$ is relatively compact in $\mathcal{C}_{d'}$ for any admissible metric $d'$. Again application of Theorem \ref{thrm:char rel compactness}, but now in opposite direction, yields that $(P_\lambda)$ is equicontinuous in $\mathcal{C}_{d'}$.
\end{proof}

\begin{prop}
	\label{prop_equivalence}
	Let $(P_\lambda)_{\lambda\in\Lambda}$ be a family of Markov operators on $(S,d)$. If $(P_\lambda)_{\lambda\in\Lambda}$ is tight, then the following are equivalent:
\begin{itemize}
	\item [(i)] For every $K\subset\MM(S)_{\BL}^+$ precompact, $(P_\lambda|_K)_{\lambda\in\Lambda}$ is equicontinuous on $K$.
	\item [(ii)] $(P_\lambda)_{\lambda\in\Lambda}$ is equicontinuous (on $S$).
\end{itemize}
\end{prop}

To prove Proposition \ref{prop_equivalence} we apply Theorem \ref{aa_kelley} and Theorem \ref{aa_b_y} to the $k$-space\\ ${(\MM^+(S)_{\BL}, \|\cdot\|_{\BL,d}^*)}$.

Now we are in a position to prove Theorem \ref{Theorem_composition}.
\begin{proof}
	(Theorem \ref{Theorem_composition})
	\label{proof_theorem_composition}
	Let $(P_\lambda)_{\lambda\in\Lambda}$ and $(Q_\gamma)_{\gamma\in\Gamma}$,with families of dual operators $(U_\lambda)_{\lambda\in\Lambda}$ and $(V_\gamma)_{\gamma\in\Gamma}$ respectively, be equicontinuous. Let $f\in\mathrm{\BL}(S,d)$. Then $\{U_\lambda f|\lambda\in\Lambda\}=\mathcal{E}$ is equicontinuous. Let $d_\mathcal{E}$ be the associated admissible metric as defined in (\ref{new_metric}) with $\mathcal{E}(f)$ replaced by $\mathcal{E}$.  Then $\mathcal{E}$ is contained in the unit ball $B_\mathcal{E}$ of $\left(\mathrm{\BL}(S,d_\mathcal{E}),\|\cdot\|_{\BL,d_{\mathcal{E}}}\right)$.\\
	As $(Q_\gamma)_{\gamma\in\Gamma}$ is an equicontinuous family for $d$, by Proposition \ref{prop:e-prop for one is e-prop for any} it is equicontinuous for any admissible metric on $S$. Hence, it is equicontinuous for $d_\mathcal{E}$. Then, by Theorem \ref{Thrm_equivalence} (iii)
	\[\mathcal{F}=\left\{V_\gamma g:\,g\in B_\mathcal{E},\, \gamma\in\Gamma\right\}\text{ is equicontinuous in } C_b(S).\]
In particular, as subset of $\mathcal{F}$,
	\[\left\{V_\gamma U_\lambda f:\,\gamma\in\Gamma, \lambda\in\Lambda\right\}\text{ is equicontinuous in } C_b(S).\]
	Hence, by Theorem \ref{Thrm_equivalence}, $(P_\lambda Q_\gamma)_{\lambda\in\Lambda, \gamma\in\Gamma}$ is equicontinuous for $d$. \\
	If $(P_\lambda)_{\lambda\in\Lambda}$ is an equicontinuous and tight family,  then Lemma \ref{precompact} implies that for any $K\subset\MM^+(S)_{\BL}$ compact, $K_Q:=\{Q_\gamma\nu|\gamma\in\Gamma, \nu\in K\}$ is precompact. Thus, $\{P_\lambda\mu|\lambda\in\Lambda, \mu\in K_Q\}=\{P_\lambda Q_\gamma\nu|\lambda\in\Lambda, \gamma\in\Gamma, \nu\in K\}\subset\MM^+(S)_{\BL}$ is precompact. In particular, this holds for  for $K=\{\nu_0\}$.
\end{proof}
	
	In the above proof of Theorem \ref{Theorem_composition} we only need assumption, that the family $(Q_\gamma)_{\gamma\in\Gamma}$ is tight.
	In case both $(P_\lambda)_{\lambda\in\Lambda}$ and $(Q_\gamma)_{\gamma\in\Gamma}$ are tight, there is an alternative way of proving Theorem \ref{Theorem_composition} using Lemma \ref{precompact}.

As a consequence of Theorem \ref{Theorem_composition} we get the following Corollary.
\begin{clry}
The composition of finite number of equicontinuous and tight families of Markov operators is equicontinuous and tight.
\end{clry}
\section{Proof of Convergence of Lie-Trotter product formula}
\label{proof_main}
\label{Lie-Trotter}
Throughout this section we assume that $(P_t^1)_{t\geq 0}$ and $(P_t^2)_{t\geq 0}$ are Markov-Feller semigroups on $S$ with dual semigroups $(U_t^1)_{t\geq 0}$, $(U_t^2)_{t\geq 0}$, respectively. 

We start by examining some consequences of Assumptions \ref{Assumption_equicontinuity_tightness}-\ref{Assumption_4_extended_commutator} formulated in Section \ref{section_main_theorems}. Introduce \[\mathcal{F}_<(\delta):=\left\{\left[P^1_{\frac{t}{n}}P^2_{\frac{t}{n}}\right]^i:n\in\NN, i\leq n-1,  t\in[0,\delta]\right\}.\]
\begin{lemma}
	\label{lemma_eq}
 The following statements hold:
 \begin{itemize}
		\item [(i)] If Assumption \ref{Assumption_equicontinuity_tightness} holds, then $\mathcal{P}^1(\delta)$ and $\mathcal{P}^2(\delta)$ are equicontinuous and tight for every $\delta>0$. 
		\item [(ii)] If $\mathcal{F}(\delta_2)$ is equicontinuous then $\mathcal{F}_<(\delta_2)$ is equicontinuous.
		\item [(iii)] $\mathcal{F}_<(\delta_2)$ is equicontinuous and tight iff $\mathcal{F}(\delta_2)$ is equicontinuous and tight.
\end{itemize}
\end{lemma}
\begin{proof} 
\begin{itemize}
		\item [(i)] Is an immediate consequence of Theorem \ref{Theorem_composition} and the semigroup property of $(P^i_t)_{t\geq 0}$.
		\item [(ii)] Let $t\in[0,\delta_2]$ and $i,n\in\NN$ such that $i\leq n-1$. Observe that $\left[P^1_{\frac{t}{n}}P^2_{\frac{t}{n}}\right]^{i}=\left[P^1_{\frac{1}{i}\frac{it}{n}}P^2_{\frac{1}{i}\frac{it}{n}}\right]^{i}$ with $\frac{it}{n}\in [0,\delta_2]$. Hence $\mathcal{F}_<(\delta_2)\subset\mathcal{F}(\delta_2)$.
		 A subset of an equicontinuous family of maps is equicontinuous. 
\item [(iii)] 
The following subsets of $\mathcal{F}_<(\delta_2)$,
\[\mathcal{F}_<^1(\delta):=\left\{P^1_{\frac{t}{n}}P^2_{\frac{t}{n}}: n\in\NN, t\in[0,\delta]\right\}\]
and
\[\mathcal{F}_<^*(\delta):=\left\{\left[P^1_{\frac{t}{n}}P^2_{\frac{t}{n}}\right]^{n-1}: n\in\NN, t\in[0,\delta]\right\}\]
are equicontinuous and tight, because $\mathcal{F}_<(\delta_2)$ is. Note that $\mathcal{F}\subset\mathcal{F}_<^1(\delta_2)\cdot\mathcal{F}_<^*(\delta_2)$. According to  
Theorem \ref{Theorem_composition}  the latter product is equicontinuous and tight. Hence $\mathcal{F}$ is equicontinuous and tight.\\
In part (ii) we observe that $\mathcal{F}_<(\delta_2)\subset \mathcal{F}(\delta_2)$, so equicontinuity and tightness of $\mathcal{F}(\delta_2)$ implies that of $\mathcal{F}_<(\delta_2)$. \qedhere
\end{itemize}\end{proof}
\begin{lemma}[Eventual equicontinuity]
	\label{lemma_eventually_equicontinuous}
	If Assumptions \ref{Assumption_equicontinuity_tightness} and \ref{Assumption_stability} hold, then for each compact $\Gamma\subset\RR_+$ there exists $N=N_\Gamma$ such that\[\mathcal{F}_\Gamma^N:=\left\{\left[P^1_{\frac{t}{n}}P^2_\frac{t}{n}\right]^n:n\in\NN, n\geq N, t\in\Gamma\right\}\] is equicontinuous.
\end{lemma}
\begin{proof}
		Let $N\in\NN$ be such that $\frac{t}{N}\leq \min(\delta_1,\delta_2)=:\delta$ for all $t\in\Gamma$. For $n\geq N$ we have, with $k:=n-N$
		\begin{equation}
		\nonumber
		\left[P^1_{\frac{t}{n}}P^2_{\frac{t}{n}}\right]^n=		\left[P^1_{\frac{1}{k}\cdot\frac{k\cdot t}{N+k}}P^2_{\frac{1}{k}\cdot\frac{k\cdot t}{N+k}}\right]^{k+N}= \left[P^1_{\frac{1}{k}\cdot\frac{k\cdot t}{N+k}}P^2_{\frac{1}{k}\cdot\frac{k\cdot t}{N+k}}\right]^{k}\left[P^1_{\frac{ t}{N+k}}P^2_{\frac{\cdot t}{N+k}}\right]^{N}.
		\end{equation}
		Since $\frac{ t}{N+k}\in[0,\delta]$ for $k\in\NN_0$ and $\mathcal{P}^1(\delta)$ and $\mathcal{P}^2(\delta)$ are equicontinuous and tight (by assumption), the family $\left\{\left[P^1_{\frac{ t}{N+k}}P^2_{\frac{\cdot t}{N+k}}\right]^{N}:k\in\NN_0, t\in\Gamma\right\}$ is equicontinuous and tight according to Theorem \ref{Theorem_composition}. The family $\left\{\left[P^1_{\frac{1}{k}\cdot\frac{k\cdot t}{N+k}}P^2_{\frac{1}{k}\cdot\frac{k\cdot t}{N+k}}\right]^k:k\in\NN, t\in\Gamma\right\}\subset\mathcal{F}(\delta_2)$ is equicontinuous by Assumption \ref{Assumption_stability}. Hence Theorem \ref{Theorem_composition} yields equicontinuity of $\mathcal{F}_\Gamma^N$.
\end{proof}
\begin{lemma}
	If Assumptions \ref{Assumption_equicontinuity_tightness} and \ref{Assumption_stability} hold and, additionally, $\mathcal{F}(\delta)$ is a tight family for some $\delta=\delta_2>0$, then  $\mathcal{F}(\delta)$ is equicontinuous and tight for any $\delta>0$.
\end{lemma}
\begin{proof}
	Let $\delta_2>0$ such that Assumption \ref{Assumption_stability} holds for $\delta_2$. Let
	\begin{equation}
	\nonumber
	\begin{array}{llll}
	&\mathcal{F}(2\delta_2)&:=&\left\{\left[P^1_{\frac{t}{n}}P^2_{\frac{t}{n}}\right]^n:t\in[0,2\delta_2], n\in\NN\right\}\\
	&&=&\underbrace{\left\{\left[P^1_{\frac{t'}{m}}P^2_{\frac{t'}{m}}\right]^{2m}:t':=\frac{t}{2}\in[0,\delta_2], m\in\NN\right\}}_{\mathcal{F}^{\text{even}}_m}\cup\underbrace{\left\{\left[P^1_{\frac{t'}{2m+1}}P^2_{\frac{t'}{2m+1}}\right]^{2m+1}:t'\in[0,\delta_2], m\in\NN\right\}}_{\mathcal{F}^{\text{odd}}_m}
	\end{array}
	\end{equation}
	 
		Due to Theorem \ref{Theorem_composition}, $\mathcal{F}^{\text{even}}_m(\delta_2)$ is an equicontinuous and tight family as a product of equicontinuous and tight families.
	\begin{equation}
	\nonumber
	\begin{array}{llll}
	&\mathcal{F}^{\text{odd}}_m(\delta_2)&=&\left\{\left[P^1_{\frac{t_m}{m}}P^2_{\frac{t_m}{m}}\right]^{2m+1}:t_m=t\cdot\frac{m}{2m+1}, t\in[0,\delta_2], m\in\NN\right\}\\
	&&\subset&\left\{\left[P^1_{\frac{t_m}{m}}P^2_{\frac{t_m}{m}}\right]\left[P^1_{\frac{t_m}{m}}P^2_{\frac{t_m}{m}}\right]^{m}\left[P^1_{\frac{t_m}{m}}P^2_{\frac{t_m}{m}}\right]^{m}:t_m=t\cdot\frac{m}{2m+1}, t\in[0,\delta_2], m\in\NN\right\}
	\end{array}
	\end{equation}
	Hence, due to Theorem \ref{Theorem_composition}, $\mathcal{F}^{\text{odd}}_m(\delta_2)$ is an equicontinuous and tight family. 
\end{proof}

\begin{lemma}
	\label{lemma_equicontinuity_of_E} Let $f\in \mathrm{\BL}(S,d)$ and  $\delta=\min(\delta_1,\delta_2)$. If Assumptions \ref{Assumption_equicontinuity_tightness} and \ref{Assumption_stability} hold, then
 $\mathcal{E}(f)$ defined by (\ref{def_E}) is equicontinuous in $C_b(S)$.
\end{lemma}
Note that $\mathcal{E}(f)$ depends on the choice of $f$. 	
	Lemma \ref{lemma_equicontinuity_of_E} is a consequence of Assumptions \ref{Assumption_equicontinuity_tightness} and \ref{Assumption_stability} and Theorem \ref{Thrm_equivalence} .

\begin{remark}
	Technically, one requires that particular subsets of $\mathcal{E}(f)$ are equicontinuous. Namely, that 
	\[\mathcal{E}_k(f)=\left\{U_{\frac{lt}{kn}}^2U_{\frac{jt}{kn}}^1\left[U^2_{\frac{t}{n}}U^1_{\frac{t}{n}}\right]^nf:n,j,l,i\in\NN,j\leq kn, i\leq n-1, l\leq kn ,t\in[0,\delta_2]\right\}\] is equicontinuous for every $k$. This seems to be quite too technical a condition. 
\end{remark}
\begin{remark}
	The commutator condition that we propose in Assumption \ref{Assumption_commutator} is weaker than the commutator conditions in \cite{Kuhnemund1}, conditions $(C)$ and $(C^*)$ in \cite{Colombo2004} and commutator condition in Proposition 3.5 in \cite{Colombo_2009}.
\end{remark}
For later reference, we present some properties of function $t\mapsto \omega(t):=\omega_f(t,\mu_0)$, that occurs in Assumptions \ref{Assumption_commutator} and \ref{Assumption_4_extended_commutator}.
	\begin{lemma}
		\label{lemma_omega} Let $\omega=\omega_f(\cdot,\mu_0):\RR_+\to\RR_+$ be a continuous, nondecreasing function such that Dini condition (\ref{equation_commutator_condition}) in Assumption \ref{Assumption_commutator} holds. Then  $\lim_{t\to 0^+}\omega(t)=0$ and for any $0<a<1$.
		\begin{itemize}
			\item [\textit{(a)}] $\sum_{n=1}^\infty \omega(a^nt)<\infty$ for all $t>0$;
			\item  [\textit{(b)}] $\lim_{t\to 0}\sum_{n=1}^\infty \omega(a^nt)=0$. 
		\end{itemize}
\end{lemma}
\begin{proof}  (a) Suppose that $\inf_{0<t<1}\omega(t)=m>0$. Then by \ref{equation_commutator_condition} in Assumption \ref{Assumption_commutator} we get
	\[\int_0^1\frac{\omega(s)}{s}ds\geq \int_0^1\frac{m}{s}ds=+\infty.\]
	So $m=0$.\\
	 From the fact that $\int_0^\sigma\frac{\omega(t)}{t}dt<+\infty$ we have 
	\begin{equation}
	\nonumber
	\begin{array}{lllll}
	\infty&>&\sum_{n=0}^\infty \int_{a^{n+1}t}^{a^nt}\frac{\omega(s)}{s}ds\geq  \sum_{n=0}^\infty\frac{\omega(a^{n+1}t)}{a^nt}(a^nt-a^{n+1}t)=\\
	&=&\sum_{n=0}^\infty \omega(a^{n+1}t)\left[1-\frac{a^{n+1}t}{a^nt}\right]=(1-a)\sum_{n=1}^\infty \omega(a^{n}t)
	\end{array}
	\end{equation} 
	This proves $(a)$.\\
	For $(b)$ let $\eps>0$. According to (a) there exists $n_0\in\NN$ such that 
	\[\sum_{n=n_0}^\infty \omega(a^n)<\frac{\eps}{2}.\]
	Moreover, because $\lim_{t\to 0^+}\omega(t)=0$, there exists $t_0\leq 1$ such that $\omega(at_0)<\frac{\eps}{2n_0}$. Then for every $0<t\leq t_0$ and $n\in\NN$, $1\leq n\leq n_0$, $\omega(a^nt)\leq \omega(at_0)\leq \frac{\eps}{2n_0}$. So
	\[\sum_{n=1}^\infty \omega(a^nt)<\sum_{n=1}^{n_0-1}\omega(a^nt)+\sum_{n=n_0}^\infty\omega(a^nt)<\frac{\eps(n_0-1)}{2n_0}+\frac{\eps}{2}<\eps.\]
\end{proof}

\label{subsection technical lemmas} To show our main result we need technical lemmas which we present in this section. Proofs of results from this section can be found in Appendix \ref{appendix_proofs} .
\begin{lemma}
	\label{lem_est} 
	The following identities hold: for fixed $k\in\NN$, $m:=kn$ and $j\leq m$. 
	\begin{itemize}
		\item [(a)]
$	P^1_{\frac{t}{m}}P^2_\frac{jt}{m}-P^2_{\frac{jt}{m}}P^1_{\frac{t}{m}}=\sum_{l=0}^{j-1}P^2_{\frac{lt}{m}}\left(P^1_{\frac{t}{m}}P^2_{\frac{t}{m}}-P^2_{\frac{t}{m}}P^1_{\frac{t}{m}}\right)P^2_{\frac{(j-1-l)t}{m}}
$
		\item [(b)]$
		P^1_{\frac{kt}{m}}P^2_\frac{kt}{m}-\left(P^1_{\frac{t}{m}}P^2_{\frac{t}{m}}\right)^k=\sum_{j=1}^{k-1}P^1_{\frac{tj}{m}}\left(P^1_{\frac{t}{m}}P^2_\frac{jt}{m}-P^2_{\frac{jt}{m}}P^1_{\frac{t}{m}}\right)P^2_{\frac{t}{m}}\left(P^1_{\frac{t}{m}}P^2_{\frac{t}{m}}\right)^{k-1-j}
$
		\item [(c)] 	$
		\left(P^1_{\frac{t}{n}}P^2_\frac{t}{n}\right)^n-\left(P^1_{\frac{t}{m}}P^2_{\frac{t}{m}}\right)^{m}=\left(P^1_{\frac{kt}{m}}P^2_\frac{kt}{m}\right)^n-\left(P^1_{\frac{t}{m}}P^2_{\frac{t}{m}}\right)^{n\cdot k}=\\
		=\sum_{i=0}^{n-1}\left(P^1_{\frac{kt}{m}}P^2_{\frac{kt}{m}}\right)^i\left(P^1_{\frac{kt}{m}}P^2_{\frac{kt}{m}}-\left(\PP\PQ\right)^k\right)\left(\PP\PQ\right)^{k\cdot(n-1-i)}
$
	\end{itemize}
\end{lemma}
Combining Lemma \ref{lem_est} (a)-(c) we get the following Corollary.
\begin{clry}
	\label{corollary_1}
	For any $n\in\NN$, $k\in\NN$ and $m:=kn$ one has
	\begin{equation}
	\nonumber
	\begin{array}{llll}
&\left(P^1_{\frac{t}{n}}P^2_{\frac{t}{n}}\right)^n-\left(P^1_{\frac{t}{m}}P^2_{\frac{t}{m}}\right)^m=\\
&	=\sum_{i=0}^{n-1}\sum_{j=1}^{k-1}\sum_{l=0}^{j-1}\left(P^1_{\frac{kt}{m}}P^2_{\frac{kt}{m}}\right)^iP^1_{\frac{jt}{m}}P^2_{\frac{lt}{m}}\left(P^1_{\frac{t}{m}}P^2_{\frac{t}{m}}-P^2_{\frac{t}{m}}P^1_{\frac{t}{m}}\right)P^2_{\frac{(j-l)t}{m}}\left(P^1_{\frac{t}{m}}P^2_{\frac{t}{m}}\right)^{k(n-i)-j-1}
	\end{array}
	\end{equation}
\end{clry}

	\begin{lemma}
		\label{lem_2 new}
		Let $f\in \BL(S,d)$ and $\mu_0\in M_0$.
		Assume that Assumptions \ref{Assumption_equicontinuity_tightness}-\ref{Assumption_4_extended_commutator} hold and put $\delta_f=\min(\delta_1,\delta_2,\delta_{3,f},\delta_{4,f})$. Then for all $t\geq 0$ and $n,k\in\NN$ such that $\frac{t}{nk}\in [0,\delta_f]$: 
			\[\left|\left\langle \left[P^1_{\frac{t}{n}}P^2_\frac{t}{n}\right]^n\mu_0-\left[P^1_{\frac{t}{kn}}P^2_{\frac{t}{kn}}\right]^{n\cdot k}\mu_0,f\right\rangle\right|\leq  C_f(\mu_0)\frac{k-1}{2}t\omega_f\left(\frac{t}{nk},\mu_0\right)\]
	\end{lemma}
We can now finally get to the proof of our main result Theorem \ref{main}, i.e
the convergence of the Lie-Trotter product formula for Markov operators. We need the lemma that yields the convergence of the subsequence of the form 	$\left\langle\left[P^1_{\frac{t}{2^n}}P^1_{\frac{t}{2^n}}\right]^{2^n}\mu_0,f\right\rangle$ for $\mu_0\in M_0$ and for every  $f\in\BL(S,d)$. Then, using this result, we will show that the sequence 	$\left\langle\left[P^1_{\frac{t}{n}}P^1_{\frac{t}{n}}\right]^{n}\mu_0,f\right\rangle$ also converges for every $f\in\BL(S,d)$. From that we can extend from $\mu_0\in M_0$ to $\mu\in\MM^+(S)$. Recall that $\delta_f:=\min(\delta_1,\delta_2,\delta_{3,f},\delta_{4,f})$. 
\begin{remark}
	The "weak" convergence in our setting is a convergence of a sequence of measures paired with a bounded Lipschitz function. Hence it differs from the "standard" definition of weak convergence (see \cite{bogachev_2} Definition 8.1.1), where the sequence of measures is paired with continuous bounded functions. However, since $\mathrm{\BL}(S,d)\simeq \MM(S)^*_{\BL}$ (see \cite{Hille2009_embedding}, Theorem 3.7) our terminology is proper from a functional analytical perspective. 
\end{remark}
\begin{lemma}
	\label{lemma subsequence}
	Let $(P^1_t)_{t\geq 0}$ and $(P^2_t)_{t\geq 0}$ be Markov semigroups such that Assumptions \ref{Assumption_equicontinuity_tightness}-\ref{Assumption_4_extended_commutator} hold.  Let $\mu_0\in M_0$ and $f\in \mathrm{\BL}(S,d)$. Then 
	the sequence $(r_n)_{n\in \NN}$ where $r_n:=\left\langle\left[P^1_{\frac{t}{2^n}}P^1_{\frac{t}{2^n}}\right]^{2^n}\mu_0,f\right\rangle$ converges for every $t\geq 0$, uniformly for $t$ in compact subsets of $\RR_+$.
\end{lemma}
\begin{proof} The case $t=0$ is trivial. So fix $t>0$. 
Let $f\in\mathrm{\BL}(S,d)$  There exists $N\in\NN$ such that $\frac{t}{2^N}\in[0,\delta_f]$. Let $i,j\in\NN, i>j\geq N$. Then $2^i=2^j\cdot 2^l$ with $l=i-j<i$. Lemma \ref{lem_2 new} yields for any $\mu_0\in M_0$, that
	\begin{equation}
	\begin{aligned}
	&\left|\left\langle\left(\left[P^1_{\frac{t}{2^i}}P^2_{\frac{t}{2^i}}\right]^{2^i}-\left[P^1_{\frac{t}{2^j}}P^2_{\frac{t}{2^j}}\right]^{2^j}\right)\mu_0, f\right\rangle\right|\\
	\leq& \sum_{l=j}^{i-1} \left|\left\langle\left(\left[P^1_{\frac{t}{2^l}}P^2_{\frac{t}{2^l}}\right]^{2^l}-\left[P^1_{\frac{t}{2^{l+1}}}P^2_{\frac{t}{2^{l+1}}}\right]^{2^{l+1}}\right)\mu_0, f\right\rangle\right|\\
	\leq& C_f(\mu_0)\frac{t}{2}\sum_{l=j}^{i-1}\omega_f\left(\frac{t}{2^{l+1}},\mu_0\right),
	\end{aligned}
	\end{equation}	
	
		with $\omega_f$ as in Assumption \ref{Assumption_commutator}. According to Lemma \ref{lemma_omega} (a),
	$\sum_{l=0}^\infty\omega_f\left(\frac{t}{2^{l+1}},\mu_0\right)<+\infty.$
So for every $\eps>0$ there exists $N'\in\NN, N'\geq N$ such that 
	$\sum_{l=j}^{i-1}\omega_f\left(\frac{t}{2^{l+1}},\mu_0\right)<\eps$
	for every $i,j\geq N$. Also, by property b) in Lemma \ref{lemma_omega}, $\omega_f\left(\frac{t}{2^{l+1}},\mu_0\right)$ can be made uniformly small, when $t$ is in a compact subset of $\RR_+$.  
Hence the sequence $(r_n)_{n\in\NN}$ is Cauchy in $\RR$, hence convergent. 
\end{proof} 
Observe that a measure $\mu\in\MM^+(S)$ is uniquely defined by its values on $f\in \BL(S,d)$. Lemma \ref{lemma subsequence} and the Banach-Steinhaus Theorem (see \cite{bogachev_1}, Theorem 4.4.3) allow us to define a positively homogeneous map $\mathbb{P}_t:M_0\to \BL(S,d)^*$ by means of
\[\label{P} \left\langle\mathbb{P}_t\mu_0,f\right\rangle:=\lim_{n\to\infty}\left\langle\left[ P_{\frac{t}{2^n}}^1P_{\frac{t}{2^n}}^2\right]^{2^n}\mu_0,f\right\rangle\]
However, according to Theorem \ref{weak_implies_strong}, $\mathbb{P}_t\mu_0\in\MM^+(S)$ for every $\mu_0\in M_0$ and 
\begin{equation}
\label{def_P}\left[P^1_\frac{t}{2^n}P^2_\frac{t}{2^n}\right]^{2^n}\mu_0\to \mathbb{P}_t\mu_0
\end{equation} strongly, in $\|\cdot\|_{\BL,d}^*$-norm. 
\begin{prop}
	\label{main theorem new}
	Let $(P^1_t)_{t\geq 0}$ and $(P^2_t)_{t\geq 0}$ be Markov semigroups such that Assumptions \ref{Assumption_equicontinuity_tightness}-\ref{Assumption_4_extended_commutator} hold. If $\mu_0\in M_0$, then for every $f\in \BL(S,d)$ and for all $t\geq 0$, 
	$\left\langle\left[P^1_{\frac{t}{n}}P^2_{\frac{t}{n}}\right]^{n}\mu_0,f\right\rangle$ converges to  $\left\langle\mathbb{P}_t\mu_0,f\right\rangle$.
\end{prop}
\begin{proof} Let  $f\in \BL(S)$, $t\geq 0$ and fix $\eps>0$. Put $\delta_f=\min(\delta_1,\delta_2,\delta_{3,f},\delta_{4,f})$. For any $l\in\NN$, using Lemma \ref{lem_2 new}, one has	
	\begin{equation}
\nonumber
\label{eq_sequence_1}
\begin{array}{llll}
&	\left|\left\langle \left[P^1_{\frac{t}{n}}P^2_{\frac{t}{n}}\right]^{n}\mu_0-\mathbb{P}_t\mu_0,f\right\rangle\right|	&\leq& \left|\left\langle \left[P^1_{\frac{t}{n}}P^2_{\frac{t}{n}}\right]^{n}\mu_0-\left[P^1_{\frac{t}{n2^l}}P^1_{\frac{t}{n2^l}}\right]^{n2^l}\mu,f\right\rangle\right|\\
&&+&\left|\left\langle\left[P^1_{\frac{t}{n2^l}}P^1_{\frac{t}{n2^l}}\right]^{n2^l}\mu_0-\left[P^1_{\frac{t}{2^l}}P^1_{\frac{t}{2^l}}\right]^{2^l}\mu_0,f\right\rangle\right|\\
&&+&\left|\left\langle\left[P^1_{\frac{t}{2^l}}P^2_{\frac{t}{2^l}}\right]^{2^l}\mu_0-\mathbb{P}_t\mu_0,f\right\rangle\right|\\
\end{array}
\end{equation}	
Pick $N$ such that for $n\geq N$ one has $\frac{t}{n}\in [0,\delta_f]$. Then 
	\begin{equation}
	\nonumber
	\label{eq_sequence_2}
\begin{array}{llll}
&	\left|\left\langle \left[P^1_{\frac{t}{n}}P^2_{\frac{t}{n}}\right]^{n}\mu_0-\mathbb{P}_t\mu_0,f\right\rangle\right|	&\leq & \sum_{i=0}^{l-1}\left|\left\langle \left[P^1_{\frac{t}{2^in}}P^2_{\frac{t}{2^in}}\right]^{2^in}\mu_0-\left[P^1_{\frac{t}{2^{i+1}n}}P^2_{\frac{t}{2^{i+1}n}}\right]^{2^{i+1}n}\mu_0,f\right\rangle\right|\\
&&&+C_f(\mu_0)\frac{n-1}{2}t\omega_f\left(\frac{t}{n2^l},\mu_0\right)\\
&&&+\left|\left\langle\left[P^1_{\frac{t}{2^l}}P^2_{\frac{t}{2^l}}\right]^{2^l}\mu_0-\mathbb{P}_t\mu_0,f\right\rangle\right|\\
&&	\leq& \sum_{i=0}^{l-1}C_f(\mu_0)\frac{1}{2}t\omega_f\left(\frac{t}{2^in},\mu_0\right)	+C_f(\mu_0)\frac{n-1}{2}t\omega_f\left(\frac{t}{n2^l},\mu_0\right)\\
&&&+\left|\left\langle\left[P^1_{\frac{t}{2^l}}P^2_{\frac{t}{2^l}}\right]^{2^l}\mu_0-\mathbb{P}_t\mu_0,f\right\rangle\right|\\
&&=&\frac{1}{2}C_f(\mu_0)t\left[\sum_{i=0}^l\omega_f\left(\frac{t}{2^in},\mu_0\right)+(n-1)\omega_f\left(\frac{t}{n2^l},\mu_0\right)\right] \\
&	&&+\left|\left\langle\left[P^1_{\frac{t}{2^l}}P^2_{\frac{t}{2^l}}\right]^{2^l}\mu_0-\mathbb{P}_t\mu_0,f\right\rangle\right|
\end{array}
\end{equation}	

	According to Proposition \ref{main theorem new} there exists $N_0$ such that for any $l\geq N_0$
\[\left|\left\langle\left[P^1_{\frac{t}{2^l}}P^2_{\frac{t}{2^l}}\right]^{2^l}\mu_0-\mathbb{P}_t\mu_0,f\right\rangle\right|<\frac{\eps}{3}. \]
Lemma \ref{lemma_omega} (b) yields $N_1\in \NN$, $N_1\geq N$ such that for every $n\geq N_1$ and $l\in\NN$,
 \[\sum_{i=0}^l\omega_f\left(\frac{t}{2^in},\mu_0\right)\leq \sum_{i=0}^\infty\omega_f\left(\frac{t}{2^in},\mu_0\right)<\left(1+\frac{1}{2}C_f(\mu_0)t\right)^{-1}\frac{\eps}{3}.\] 
 Since $\omega_f(s,\mu_0)\downarrow 0$ as $s\downarrow 0$, for every $n\geq N_1$, there exists $l_n\geq N_0$ such that
 \[\omega_f\left(\frac{t}{n2^{l_n}},\mu_0\right)<\frac{1}{n-1}\left(1+\frac{1}{2}tC_f(\mu_0)\right)^{-1}\frac{\eps}{3}.\]
So by choosing $l=l_n$ in the above derivation, we get that
 \[\left|\left\langle \left[P^1_{\frac{t}{n}}P^2_{\frac{t}{n}}\right]^{n}\mu_0-\mathbb{P}_t\mu_0,f\right\rangle\right|<\eps\text{ for every } n\geq N_1.\]
\end{proof}
Next lemma shows that once convergence of $\left\langle\left[P_{\frac{t}{n}}^1P_{\frac{t}{n}}^2\right]^n\mu_0,f\right\rangle$ is established for $\mu_0\in M_0$ then we have convergence for all $\mu\in\MM^+(S)$. 
\begin{lemma}
	\label{lemma_Cauchy_seq}
 Assume that Assumptions \ref{Assumption_equicontinuity_tightness}-\ref{Assumption_4_extended_commutator} hold. Then for every $\mu\in \MM^+(S)$ and $t\geq 0$,  $\left(\left[P^1_{\frac{t}{n}}P^2_{\frac{t}{n}}\right]^{n}\mu\right)_{n\in\NN}$ is a Cauchy sequence in $\mu\in\MM^+(S)$ for $\|\cdot\|_{\BL,d}^*$.\\
\end{lemma}
\begin{proof} Let $\mu\in\MM^+(S)$. Let $\epsilon>0$. By Assumption \ref{Assumption_stability}, $\mathcal{F}(\delta)$ is an equicontinuous family. Thus there exists $\delta_\epsilon>0$ such that 
	\[\left\|\left[P^1_{\frac{t}{n}}P^2_{\frac{t}{n}}\right]^{n}\mu-\left[P^1_{\frac{t}{n}}P^2_{\frac{t}{n}}\right]^{n}\nu\right\|_{\BL,d}^*< \epsilon/3\] for all $\nu\in\MM^+(S)$ such that $\|\mu-\nu\|_{\BL,d}^*<\delta_\epsilon$. As $M_0\subset\MM^+(S)$ dense, there exists $\mu_0\in M_0$ such that $||\mu-\mu_0\|_{\BL,d}^*<\delta_\epsilon$. Then 
\begin{eqnarray}
\label{eq_equicontinuity}
\begin{array}{llll}
&\left\|\left[P^1_{\frac{t}{n}}P^2_{\frac{t}{n}}\right]^{n}\mu-\left[P^1_{\frac{t}{m}}P^2_{\frac{t}{m}}\right]^{m}\mu\right\|_{\BL,d}^*&\leq& \left\|\left[P^1_{\frac{t}{n}}P^2_{\frac{t}{n}}\right]^{n}\mu-\left[P^1_{\frac{t}{n}}P^2_{\frac{t}{n}}\right]^{n}\mu_0\right\|_{\BL,d}^*\\
&&+&\left\|\left[P^1_{\frac{t}{n}}P^2_{\frac{t}{n}}\right]^{n}\mu_0-\left[P^1_{\frac{t}{m}}P^2_{\frac{t}{m}}\right]^{m}\mu_0\right\|_{\BL,d}^*\\
&&+&\left\|\left[P^1_{\frac{t}{m}}P^2_{\frac{t}{m}}\right]^{m}\mu_0-\left[P^1_{\frac{t}{m}}P^2_{\frac{t}{m}}\right]^{m}\mu\right\|_{\BL,d}^*
\end{array}
\end{eqnarray}
According to Proposition \ref{main theorem new} and Theorem \ref{weak_implies_strong}, there exists $N\in\NN$ such that for $n,m\geq N$,
\[\left\|\left[P^1_{\frac{t}{n}}P^2_{\frac{t}{n}}\right]^{n}\mu_0-\left[P^1_{\frac{t}{m}}P^2_{\frac{t}{m}}\right]^{m}\mu_0\right\|_{\BL,d}^*<\epsilon/3.\]
Hence for $n,m\geq N$, we obtain for \ref{eq_equicontinuity} that
\[\left\|\left[P^1_{\frac{t}{n}}P^2_{\frac{t}{n}}\right]^{n}\mu-\left[P^1_{\frac{t}{m}}P^2_{\frac{t}{m}}\right]^{m}\mu\right\|_{\BL,d}^*<\frac{\epsilon}{3}+\frac{\epsilon}{3}+\frac{\epsilon}{3}=\epsilon\] which proves that $\left(\left[P^1_{\frac{t}{n}}P^2_{\frac{t}{n}}\right]^{n}\mu\right)_n$ is a Cauchy sequence.
\end{proof}
Lemma \ref{lemma_Cauchy_seq} allows us to define for $\mu\in\MM^+(S)$ and $t\in [0,\delta]$
\[\bar{\mathbb{P}}_t\mu:=\lim_{n\to\infty}\left[P^1_{\frac{t}{n}}P^2_{\frac{t}{n}}\right]^n\mu\] as a limit in $\MM^+(S)_{\BL}$. Then $\bar{\mathbb{P}}_t\mu_0=\mathbb{P}_t\mu_0$ for $\mu_0\in M_0$, according to Proposition \ref{main theorem new}.

Thus, as a consequence of Lemma \ref{lemma_Cauchy_seq} we have proven the first part of Theorem \ref{main}.

	Concerning the second part of the proof: the arguments in the proofs of the lemmas and propositions that together finish the proof of Theorem \ref{main}, show upon inspection that in case where stronger versions of Assumptions \ref{Assumption_commutator} and \ref{Assumption_4_extended_commutator} hold, then immediately $\|\cdot\|_{\BL,d}^*$-norm estimates can be obtained. That is, if is Assumptions \ref{Assumption_commutator} and \ref{Assumption_4_extended_commutator} a single $\delta_{3,f}$, $\delta_{4,f}$. $C_f(\mu_0)$ and $\omega_f(\cdot,\mu_0)$ can be chosen to hold uniformly for $f$ in the unit ball of $\BL(S,d)$, then one obtains Theorem \ref{main} (ie. norm-convergence of the Lie-Trotter product) without the need of Theorem \ref{weak_implies_strong}. Then one easily checks that convergence is uniform in $t$ in compact subsets of $\RR_+$. In fact for $\mu\in M_0$ this result is captured in the preceding remarks. Let $\Gamma\subset \RR_+$ be compact. According to Lemma \ref{lemma_eventually_equicontinuous} $\mathcal{F}_\Gamma^N$ is equicontinuous for $N$ sufficiently large. Then all estimates in the proof of Lemma \ref{lemma_Cauchy_seq} can be made uniformly in $t\in\Gamma$. 
	
 Moreover, in the situation described above, the rate of convergence of the Lie-Trotter product is controlled by properties of $\omega(\cdot,\mu_0)$, according to the proof of Proposition \ref{main theorem new}.

	\section{Properties of the limit}
	Let us now analyze properties of the limit operator family $(\mathbb{\overline{P}}_t)_{t\geq 0}$ as obtained by the Lie-Trotter product formula. First we show that $\mathbb{\overline{P}}_t$ is a Feller operator, i.e. it is continuous on $\MM^+(S)$ for $\|\cdot\|_{\mathrm{\BL,d}}^*$:
	\subsection{Feller property}

\begin{lemma}
	\label{convergence}
	Let  $(P^1_t)_{t\geq 0}$ and $(P^2_t)_{t\geq 0}$  be semigroups of regular Markov-Feller operators that satisfy Assumptions \ref{Assumption_equicontinuity_tightness}-\ref{Assumption_4_extended_commutator}. Let $(\mu_n)_{n\in\NN}\subset\MM^+(S)$ and $\mu^*\in\MM^+(S)$ be such that $\mu_n\to \mu^*$ in $\MM^+(S)_{\BL}$ as $n\to \infty$.  Then $\left[P^1_{\frac{t}{n}}P^2_{\frac{t}{n}}\right]^n\mu_n\to\mathbb{\overline{P}}_t\mu^*$ in $\MM^+(S)_{\BL}$ for $t\in[0,\delta_2]$.
\end{lemma}
\begin{proof}
	Let $\epsilon>0$.  From Assumption \ref{Assumption_stability} (stability) we get that there exists $\delta_\epsilon>0$ such that 
	\[\left\|\left[P^1_{\frac{t}{n}}P^2_{\frac{t}{n}}\right]^n\mu-\left[P^1_{\frac{t}{n}}P^2_{\frac{t}{n}}\right]^n\mu^*\right\|^*_{\BL,d}<\eps/2\] for every $\nu\in\MM^+(S)$ such that $\|\mu-\mu^*\|_{\BL,d}^*<\delta_\epsilon$ for all $t\in[0,\delta_2]$.\\
	Since $\mu_n\to\mu^*$, there exists $N_0\in\NN$ such that \[\|\mu_n-\mu^*\|_{\BL,d_{\mathcal{E}(f)}}^*<\delta_\epsilon\] for all $n\geq N_0$.
	From Theorem \ref{main} we know that  there exists $N_1\in\NN$ such that for every $n\geq N_1$
	\[\left\|\left[P^1_{\frac{t}{n}}P^2_{\frac{t}{n}}\right]^n\mu^*-\mathbb{\overline{P}}_t\mu^*\right\|_{\BL,d}<\epsilon/2\]
	Then for $n\geq N:=\max(N_0,N_1)$, 
	\begin{equation}
	\nonumber
	\begin{array}{llll}
	&\left\|\left[P^1_{\frac{t}{n}}P^2_{\frac{t}{n}}\right]^n\mu_n-\mathbb{\overline{P}}_t\mu^*\right\|_{\BL,d}^*&\leq& \left\|\left[P^1_{\frac{t}{n}}P^2_{\frac{t}{n}}\right]^n\mu_n-\left[P^1_{\frac{t}{n}}P^2_{\frac{t}{n}}\right]^n\mu^*\right\|_{\BL,d}^*\\
	&&+&\left\|\left[P^1_{\frac{t}{n}}P^2_{\frac{t}{n}}\right]^n\mu^*-\mathbb{\overline{P}}_t\mu^*\right\|_{\BL,d_{\mathcal{E}(f)}}^*<\epsilon.\quad \qedhere
	\end{array}
	\end{equation}
\end{proof}
\begin{prop}
	\label{prop_semigroup_k}
	If Assumptions \ref{Assumption_equicontinuity_tightness}-\ref{Assumption_4_extended_commutator} then for all $k\in\NN, t\geq 0$
	\[\mathbb{\overline{P}}_{kt}\mu=\mathbb{\overline{P}}^k_{t}\mu\text{ for all } \mu\in\MM^+(S).\]
	In particular, $\overline{\mathbb{P}}_t\overline{\mathbb{P}}_s\mu=\overline{ {\mathbb{P}}}_{t+s}\mu$ for all $t,s\geq 0$ such that $\frac{t}{s}\in\mathbb{Q}$.
\end{prop}
\begin{proof}Let $\mu\in\MM^+(S)$.  Let $\epsilon>0$. 
	Without loss of generality we can assume that $t\in[0,\delta_2]$. For $k=1$ the statement is obviously true. Assume it has been proven for $k$. We now show it holds for $k+1$ as well. As we know that the limit of the Lie-Trotter product exists (Theorem \ref{main}), we can consider in the limit any subsequence. Take $n=(k+1)m$, $m\to\infty$:
	\[\overline{\mathbb{P}}_{(k+1)t}\mu=\lim_{m\to\infty} \left[P^1_{\frac{t}{m}}P^2_{\frac{t}{m}}\right]^{(k+1)m}\mu=\lim_{m\to\infty}\left[P^1_{\frac{t}{m}}P^2_{\frac{t}{m}}\right]^{m} \left(\left[P^1_{\frac{t}{m}}P^2_{\frac{t}{m}}\right]^{km}\mu\right).\]
	Hence there exists $N_0\in\NN$ such that for all $m>N_0$,
	\begin{align*}
	 \left\|\overline{\mathbb{P}}_{(k+1)t}\mu-\left[P^1_{\frac{t}{m}}P^2_{\frac{t}{m}}\right]^{m} \left(\left[P^1_{\frac{t}{m}}P^2_{\frac{t}{m}}\right]^{km}\mu\right)\right\|_{BL,d}^*<\frac{\epsilon}{3}
	\end{align*}
Since by assumption $\left[P^1_{\frac{t}{m}}P^2_{\frac{t}{m}}\right]^{km}\mu\to \overline{\mathbb{P}}_{kt}\mu$, Lemma \ref{convergence} yields that there exists $N_1\geq N_0$ such that for $m\geq N_1$:
\begin{align*}
\left\|\left[P^1_{\frac{t}{m}}P^2_{\frac{t}{m}}\right]^{m} \left(\left[P^1_{\frac{t}{m}}P^2_{\frac{t}{m}}\right]^{km}\mu\right)-\left[P^1_{\frac{t}{m}}P^2_{\frac{t}{m}}\right]^{m}\overline{ {\mathbb{P}}}_{kt}\mu\right\|_{BL,d}^*<\frac{\epsilon}{3}.
\end{align*}
Also, by Theorem \ref{main} we get $N_2\geq N_1$ such that for every $m\geq N_2$
\begin{align*}
\left\|\left[P^1_{\frac{t}{m}}P^2_{\frac{t}{m}}\right]^{m} \overline{{\mathbb{P}}}_{kt}\mu-\overline{\mathbb{P}}^{k+1}_t\mu\right\|_{BL,d}^*<\frac{\epsilon}{3}
\end{align*}
Hence for $m\geq N_2$, 
\begin{align*}
\left\|\overline{\mathbb{P}}_{(k+1)t}\mu-\overline{\mathbb{P}}^{k+1}_t\mu\right\|_{BL,d}^*&\leq \left\|\overline{\mathbb{P}}_{(k+1)t}\mu-\left[P^1_{\frac{t}{m}}P^2_{\frac{t}{m}}\right]^{m} \left(\left[P^1_{\frac{t}{m}}P^2_{\frac{t}{m}}\right]^{km}\mu\right)\right\|_{BL,d}^*\\
&+\left\|\left[P^1_{\frac{t}{m}}P^2_{\frac{t}{m}}\right]^{m} \left(\left[P^1_{\frac{t}{m}}P^2_{\frac{t}{m}}\right]^{km}\mu\right)-\left[P^1_{\frac{t}{m}}P^2_{\frac{t}{m}}\right]^{m} {\mathbb{P}}_{kt}\mu\right\|_{BL,d}^*\\
&+\left\|\left[P^1_{\frac{t}{m}}P^2_{\frac{t}{m}}\right]^{m} {\mathbb{P}}_{kt}\mu-\overline{\mathbb{P}}^{k+1}_t\mu\right\|_{BL,d}^*<\epsilon.
\end{align*}
If $t,s>0$ are such that  $\frac{t}{s}\in\mathbb{Q}$, then there exist $m,r\in\NN$: $rt=ms$. Hence, by the first part, 
\[\mathbb{\overline{P}}_{t+s}\mu=\mathbb{\overline{P}}_{(m+r)\cdot\frac{s}{r}}\mu=\mathbb{\overline{P}}^{(m+r)}_{\frac{s}{r}}\mu=\mathbb{\overline{P}}^{m}_{\frac{s}{r}}\mathbb{\overline{P}}^{r}_{\frac{s}{r}}\mu=\overline{\mathbb{P}}_t\overline{\mathbb{P}}_s\mu. \quad\qedhere\]
\end{proof}
\begin{prop} 
	\label{prop_Feller_limit}
	$\mathbb{\overline{P}}_t:\MM^+(S)_{\BL}\to		\MM^+(S)_{\BL}$ is continuous for all $t\geq 0$. 
\end{prop} 
	\begin{proof}First we will get the result for $t\in [0,\delta_2]$.

	Let $\mu\in\MM^+(S)$ and $\epsilon>0$. By Assumption \ref{Assumption_stability}, there exists $\delta_\epsilon>0$ such that
	\begin{equation}
	\label{eq_prop_Feller}
	\left\|\left[P^1_\frac{t}{n}P^2_{\frac{t}{n}}\right]^n\mu-\left[P^1_\frac{t}{n}P^2_{\frac{t}{n}}\right]^n\nu\right\|_{\BL,d}^*<\frac{\epsilon}{2}
	\end{equation}
	for every $\nu\in\MM^+(S)$ such that $\|\mu-\nu\|_{\BL,d}^*<\delta_\epsilon$ and all $n\in\NN, t\in [0,\delta_2]$.\\
	Then, by taking the limit $n\to\infty$ in (\ref{eq_prop_Feller}), using Theorem \ref{main},
	\[\left\|\mathbb{\overline{P}}_t\mu-\mathbb{\overline{P}}_t\nu\right\|_{\BL,d}^*\leq \frac{\epsilon}{2}<\epsilon\]
	for all $\mu,\nu\in\MM^+(S)$ such that $\|\mu-\nu\|_{\BL,d}^*<\delta_\epsilon$. So $\mathbb{\overline{P}}_t$ is continuous for all $t\in[0,\delta_2]$. 
	
	Now we can use Proposition \ref{prop_semigroup_k} to extend the result to all $t\geq 0$. 
\end{proof}
In the proof we actually show more, which we formulate as a corollary.
\begin{clry}
	\label{Corollary_equicontinuity}
The family	$\overline{\mathcal{P}}(\delta)=\left\{\overline{\mathbb{P}}_t:t\in [0,\delta]\right\}$ is equicontinuous for every $0<\delta\leq \delta_2$. 
\end{clry}

\subsection{Semigroup property}
Let us now analyze the full semigroup property of the limit. 
Recall Proposition \ref{prop_semigroup_k}. The extension to all pairs $t,s\in\RR_+$ of the semigroup property is not obvious. We do not assume any continuity of Markov semigroups. However,
\begin{prop}
	\label{Proposition_semigroup}
	Assume that Assumptions \ref{Assumption_equicontinuity_tightness}-\ref{Assumption_4_extended_commutator} hold and additionally that $t\mapsto P_t^i\mu:\RR_+\to\MM^+(S)_{\BL}$ are continuous for $i=1,2$ and all $\mu\in\MM^+(S)$. Then $(\overline{\mathbb{P}}_t)_{t\geq 0}$ is strongly continuous and it is a semigroup. 
\end{prop}
\begin{proof}
	Put $\mathbb{Q}^{n}_{t}:=\left[P^1_{\frac{t}{n}}P^2_{\frac{t}{n}}\right]^n$. If $\mu_0\in M_0$, then by strong continuity of the semigroup $(P_t^i)_{t\geq 0}$  on $\MM^+(S)$, we obtain that $F_n:\RR_+\to\RR:t\mapsto \langle \mathbb{Q}^{n}_{t}\mu_0,f\rangle$ is continuous for all $n\in\NN$. According to Lemma \ref{lemma subsequence}, $F_{2^N}$ converges uniformly on compact subsets of $\RR_+$ to $t\mapsto \langle \overline{ {\mathbb{P}}}\mu_0,f \rangle$. Hence the latter function is continuous on $\RR_+$.
	
	Now, first take $t^*\in \left[0,\delta_2\right)$  and $(t_k)_{k}\subset\left[0,\delta_2\right) $ such that $(t_k)_k\to t^*$. Let $\mu\in\MM^+(S)$ and $\epsilon>0$. Since the family $\overline{ {\mathcal{P}}}(\delta_2)$ is equicontinuous (Corollary \ref{Corollary_equicontinuity}), there exists $\delta_\epsilon>0$ such that for all $\nu\in\MM^+(S)$ with $\|\mu-\nu\|_{\BL,d}^*<\delta_\epsilon$,
	\[\left\|\overline{ {\mathbb{P}}}_t\mu-\overline{ {\mathbb{P}}}_t\nu\right\|_{\BL,d}^*<\frac{\epsilon}{3(1+\|f\|_{\BL,d})}\quad \text{ for all } t\in [0,\delta_2].\]
	$M_0$ is dense in $\MM^+(S)$. So there exists $\nu_0\in M_0$ such that $\|\mu-\mu_0\|_{\BL,d}^*<\delta_\epsilon$. Then
	\begin{align*}
	\left|\left\langle\overline{ {\mathbb{P}}}_{t^*}\mu-\overline{ {\mathbb{P}}}_{t_k}\mu,f\right\rangle\right|&\leq \left\|\overline{ {\mathbb{P}}}_{t^*}\mu-\overline{ {\mathbb{P}}}_{t^*}\mu_0\right\|_{\BL,d}^*\cdot\|f\|_{\BL,d}\\
	&+\left|\left\langle\overline{ {\mathbb{P}}}_{t^*}\mu_0-\overline{ {\mathbb{P}}}_{t_k}\mu_0,f\right\rangle\right|+ \left\|\overline{ {\mathbb{P}}}_{t_k}\mu_0-\overline{ {\mathbb{P}}}_{t_k}\mu\right\|_{\BL,d}^*\cdot\|f\|_{\BL,d}\\
	&<\frac{\epsilon}{3}+\frac{\epsilon}{3}+\frac{\epsilon}{3}=\epsilon,
	\end{align*}
	when $k\geq N$ such that $\left|\left\langle\overline{ {\mathbb{P}}}_{t^*}\mu_0-\overline{ {\mathbb{P}}}_{t_k}\mu_0,f\right\rangle\right|<\frac{\epsilon}{3}$ for all $k\geq N$. So, by Theorem \ref{weak_implies_strong}, $t\mapsto \overline{ {\mathbb{P}}}_t\mu$ is continuous on $\left[0,\delta_2\right)$.

Now we show that continuity of $t\mapsto \overline{ {\mathbb{P}}}_t\mu$ on $\left[0,m\delta_2\right)$ implies continuity on $\left[0,(m+1)\delta_2\right)$. Let $t^*\in \left[0,(m+1)\delta_2\right)$ and $t_k\in \left[0,(m+1)\delta_2\right)$ such that $t_k\to t^*$. According to Proposition \ref{prop_semigroup_k},
\begin{align*}
\overline{ {\mathbb{P}}}_{t_k}\mu&=\overline{ {\mathbb{P}}}_{\frac{t_k}{m+1}}\left(\overline{ {\mathbb{P}}}_{\frac{mt_k}{m+1}}\mu\right)=\overline{ {\mathbb{P}}}_{\frac{t_k}{m+1}}\left[\overline{ {\mathbb{P}}}_{\frac{mt_k}{m+1}}\mu-\overline{ {\mathbb{P}}}_{\frac{mt^*}{m+1}}\mu\right]+\overline{ {\mathbb{P}}}_{\frac{t_k}{m+1}}\cdot\overline{ {\mathbb{P}}}_{\frac{mt^*}{m+1}}\mu.
\end{align*}
Because $\frac{t_k}{m+1}\in \left[0,\delta_0\right)$, $\overline{ {\mathcal{P}}}(\delta_2)$ is equicontinuous and $\overline{ {\mathbb{P}}}_{\frac{mt_k}{m+1}}\mu\to\overline{ {\mathbb{P}}}_{\frac{mt^*}{m+1}}\mu$ by assumption, the first term can be made arbitrarily small for sufficiently large $k$. The second term converges to $\overline{ {\mathbb{P}}}_{\frac{t^*}{m+1}}\cdot\overline{ {\mathbb{P}}}_{\frac{mt^*}{m+1}}\mu$, which equals $\overline{ {\mathbb{P}}}_{t^*}\mu$ by Proposition \ref{prop_semigroup_k}. So indeed, $t\mapsto \overline{ {\mathbb{P}}}_t\mu$ is continuous on $\left[0,(m+1)\delta_2\right)$. We conclude that $t\mapsto \overline{ {\mathbb{P}}}_t\mu$ is continuous on $\RR_+$. According to Proposition \ref{prop_semigroup_k}, $\overline{ {\mathbb{P}}}_t\overline{ {\mathbb{P}}}_s\mu=\overline{ {\mathbb{P}}}_{t+s}\mu$ for all $t,s\in\RR_+$ such that $\frac{t}{s}\in\mathbb{Q}$. Because $t\mapsto \overline{ {\mathbb{P}}}_t\mu$ is continuous, the semigroup property must hold for all $t,s\in\RR_+$. 
\end{proof}

We say that Markov semigroup is \textbf{stochastically continuous at 0} if $\lim_{h\searrow 0}P_h\mu=\mu$ for every  $\mu\in\MM^+(S)_{\BL}$.  Stochastic continuity at 0 implies \textbf{right-continuity} at every $t_0\geq 0$, but not left-continuity.\\
Next result shows together with equicontinuity, stochastic continuity at 0 implies strong continuity.
\begin{prop}
	\label{prop1}
	Let $(P_t)_{t\geq 0}$ be a Markov-Feller semigroup. Assume that there exists $\delta>0$ such that $(P_t)_{t\in[0,\delta]}$ is equicontinuous. If $(P_t)_{t\geq 0}$ is stochastically continuous at 0, then it is strongly continuous. 
\end{prop}
\begin{proof}
	
	$(P_t)_{t\in[0,\delta]}$ is equicontinuous and  $P_{t'}$ is Feller for all $t'\geq 0$. Consequently, $(P_t)_{t\in[t', t'+\delta]}$ is an equicontinuous family for every $t'\in\RR_+$. Hence $(P_t)_{t\in [0,T]}$ is equicontinuous for every $T\in \RR_+$. So, if $\eps>0$, there exists an open neighbourhood $U$ in $\MM^+(S)$ of $\mu$ such that 
	\[\|P_t\nu-P_t\mu\|_{\BL}^*<\eps\] for every $\nu\in U$. Let $t_0>0$. \\
	From the fact, that $(P_t)_{t\geq 0}$ is (strongly) stochastically continuous at 0, there exists $\delta>0$ such that for every $0<h<\delta$, $P_h\mu\in U$. Then, from the fact that
	\[\|P_{t_0}\mu-P_{t_0-h}\mu\|_{\BL}^*=\|P_{t_0-h}\mu-P_{t_0-h}P_{h}\mu\|_{\BL}^*,\]
	we get
	\[\|P_{t_0-h}\mu-P_{t_0}\mu\|_{\BL}^*<\eps  \text { for all } \, 0<h<\delta.\]
	So  $t\mapsto P_t\mu$ is also left-continuous at every $t_0> 0$. 
\end{proof}
\begin{clry}
	If $(P_t)_{t\geq 0}$ is stochastically continuous and $(P_t)_{t\in[0,\delta]}$ is equicontinuous, then $(P_t)_{t\in [0,T]}$ is tight for every $T>0$. 
\end{clry}
\begin{remark}
	\label{rem_1}
	From Proposition \ref{prop1} we can conclude that a Markov semigroup that is stochastically continuous at 0 but not strongly continuous, cannot be equicontinuous. 
\end{remark}

\subsection{Symmetry}
We prove that, if the family $\mathcal{P}^1(\delta)$ is tight-as we assume in Assumption \ref{Assumption_equicontinuity_tightness}- then the limit does not depend on the order in which we start switching semigroups $(P^1_t)_{t\geq 0}$ and $(P^2_t)_{t\geq 0}$.

Now let us prove the following lemma. 
\begin{lemma}
	\label{lemma_swapping_semigroups}Let  $(P^1_t)_{t\in\TT}$ and $(P^2_t)_{t\in\TT}$  be semigroups of regular Markov-Feller operators. Let $n\in\NN$, $t\in\RR_+$. Then
	
	\begin{align}
	\label{eq_swapped}
	\left(P^1_{t}P^2_{t}\right)^n-\left(P^2_{t}P^1_{t}\right)^n&=\sum_{i=0}^{n-1}(P_t^2P_t^1)^{n-i-1}C_{t,t}^{1,2}(P_t^1P_t^2)^i \\
	&		=\sum_{i=0}^{n-1}(P_t^1P_t^2)^{n-i-1}C_{t,t}^{1,2}(P_t^2P_t^1)^i
	\end{align}
	
	where $C_{s,t}^{i,j}=P^i_sP^j_t-P^j_tP^i_s$. 
\end{lemma}
\begin{proof}
	We prove (\ref{eq_swapped}) by induction. Let $L_n$ denote the left hand side in equality (\ref{eq_swapped}), $R_n$ the right hand side. Obviously $L_1=R_1$.
	Assume that $L_{n-1}=R_{n-1}$. Then:
	\begin{align}
	\nonumber
	\begin{array}{llll}
	&L_n&=&\left(P^1_{s}P^2_{s}\right)^n-\left(P^2_{s}P^1_{s}\right)^n=\\
	&&=&\left[\left(P^1_{s}P^2_{s}\right)^{n-1}-\left(P^2_{s}P^1_{s}\right)^{n-1}\right]P_s^1P_s^2+\left(P^2_{s}P^1_{s}\right)^{n-1}P_s^1P_s^2-\left(P^2_{s}P^1_{s}\right)^n=\\
	&&=&\left[\sum_{i=0}^{n-2}(P_s^2P_s^1)^{n-i-2}C_{s,s}^{1,2}(P_s^1P_s^2)^i\right]P_s^1P_s^2+\left(P^2_{s}P^1_{s}\right)^{n-1}\left(P_s^1P_s^2-P^2_{s}P^1_{s}\right)=\\
	&&=&\sum_{i=0}^{n-2}(P_s^2P_s^1)^{n-i-2}C_{s,s}^{1,2}(P_s^1P_s^2)^{i+1}+\left(P^2_{s}P^1_{s}\right)^{n-1}C_{s,s}^{1,2}=\\
	&&=&\sum_{i=0}^{n-1}(P_s^2P_s^1)^{n-i-1}C_{s,s}^{1,2}(P_s^1P_s^2)^{i}=R_n.\quad \qedhere
	\end{array}
	\end{align}\end{proof}

Next we prove that the limit of the switching scheme does not depend on the order of switched semigroups in the product formula. 
\begin{prop}Let  $(P^1_t)_{t\geq 0}$ and $(P^2_t)_{t\geq 0}$  be semigroups of Markov operators for which Assumptions \ref{Assumption_equicontinuity_tightness}-\ref{Assumption_4_extended_commutator} hold and additionally, that Assumption \ref{Assumption_stability} holds for $(P^1_t)_{t\geq 0}$ and $(P^2_t)_{t\geq 0}$ swapped. Let $\mu\in\MM^+(S)$. Then 
		\[\lim_{n\to\infty}\left[P^1_{\frac{t}{n}}P^2_{\frac{t}{n}}\right]^n\mu=\lim_{n\to\infty}\left[P^2_{\frac{t}{n}}P^1_{\frac{t}{n}}\right]^n\mu\]
	\end{prop}
\begin{proof}
	Let $t\in \RR_+$, $\mu_0\in M_0$, $f\in \BL(S,d)$ and fix $\eps>0$. There exists $N\in\NN$ such that $\frac{t}{N}\leq \delta$, where $\delta=\min(\delta_{3,f},\delta_{4,f})$. Since $(P^1_t)_{t\geq 0}$ and $(P^2_t)_{t\geq 0}$ are equicontinuous, they consist of Feller operators necessarily. According to Lemma \ref{lemma_swapping_semigroups}, for $n\geq N$
	\begin{equation}
\nonumber
\begin{array}{llll}
&\left|\left\langle \left[P_\frac{t}{n}^1P_\frac{t}{n}^2\right]^n\mu_0-\left[P_\frac{t}{n}^2P_\frac{t}{n}^1\right]^n\mu_0,f\right\rangle\right|&=&\left|\left\langle \sum_{i=0}^{n-1}\left[P_\frac{t}{n}^1P_\frac{t}{n}^2\right]^{n-i-1}C_{\frac{t}{n},\frac{t}{n}}^{1,2}\left[P_\frac{t}{n}^2P_\frac{t}{n}^1\right]^i\mu_0,f\right\rangle\right|\\
&&\leq&\sum_{i=0}^{n-1}\left|\left\langle C_{\frac{t}{n},\frac{t}{n}}^{1,2}\left[P_\frac{t}{n}^2P_\frac{t}{n}^1\right]^i\mu_0,\left[U_\frac{t}{n}^2U_\frac{t}{n}^1\right]^{n-i-1}f\right\rangle\right|\\
&&\leq &\sum_{i=0}^{n-1}\left\|C_{\frac{t}{n},\frac{t}{n}}^{1,2}\left[P_\frac{t}{n}^2P_\frac{t}{n}^1\right]^i\mu_0\right\|_{\BL,d_{\mathcal{E}(f)}}^*\cdot\left\|\left[U_\frac{t}{n}^2U_\frac{t}{n}^1\right]^{n-i-1}f\right\|_{\BL,d_{\mathcal{E}(f)}}\\
&&\leq&\sum_{i=0}^{n-1}\frac{t}{n}\omega_f\left(\frac{t}{n},\left[P_\frac{t}{n}^2P_\frac{t}{n}^1\right]^i\mu_0\right)
\\
&&\leq&C_f(\mu_0)t\omega_f\left(\frac{t}{n},\mu_0\right),
\end{array}
\end{equation}	
because $\left[U_\frac{t}{n}^2U_\frac{t}{n}^1\right]^{n-i-1}f\in\mathcal{E}(f)$. 

As $t$ is fixed and $\lim_{s\to 0}\omega_f(s,\mu_0)=0$, we obtain for every $f\in \BL(S,d)$ and $\mu_0\in M_0$ 
	\begin{equation}
\nonumber
\begin{array}{llll}
&\lim_{n\to\infty}\left|\left\langle \left[P_\frac{t}{n}^1P_\frac{t}{n}^2\right]^n\mu_0-\left[P_\frac{t}{n}^2P_\frac{t}{n}^1\right]^n\mu_0,f\right\rangle\right|&=&0\
\end{array}
\end{equation}
 Then, by Theorem \ref{weak_implies_strong}, it also converges in norm. Hence, 
\[\left\|\left[P_\frac{t}{n}^1P_\frac{t}{n}^2\right]^n\mu_0-\left[P_\frac{t}{n}^2P_\frac{t}{n}^1\right]^n\mu_0\right\|_{\BL}^*\to 0\text{ as }n\to \infty.\]
Define $\hat{\mathbb{P}}_t\mu:=\lim_{n\to\infty}\left[P^2_{\frac{t}{n}}P^1_{\frac{t}{n}}\right]^n\mu$, for $\mu\in \MM^+(S)$. Since by assumption Assumption \ref{Assumption_stability} holds witj $P^1_t$ ans $P^2_t$ swapped, Proposition \ref{prop_Feller_limit} holds for $\hat{\mathbb{P}}_t$ as well: both $\mathbb{\overline{P}}_t$ and $\hat{\mathbb{P}}_t$ are continuous on $\MM^+(S)$. Since $M_0$ is a dense subset of $\MM^+(S)_{\BL}$ and $\mathbb{\overline{P}}_t\mu_0=\hat{\mathbb{P}}_t\mu_0$ for $\mu_0\in M_0$, we obtain $\mathbb{\overline{P}}_t=\hat{\mathbb{P}}_t$ on $\MM^+(S)$. 
\end{proof}
\section{Relation to literature}
\label{section relations to literature}
We shall now show that Theorem \ref{main} is a generalization of existing results. We start with K\"uhnemund and Wacker \cite{Kuhnemund_Wacker} approach and show in details that their result follows from Theorem \ref{main}. Then we provide proof that also the Proposition 3.5 in Colombo-Guerra \cite{Colombo_2009} follows from Theorem \ref{main}. 
\subsection{K\"uhnemund-Wacker}
\label{KW_case}
K\"uhnemund and Wacker \cite{Kuhnemund_Wacker} provided conditions for $C_0$-semigroups that ensure convergence of the Lie-Trotter product. Their setting is the following:

 Let $(T(t))_{t\geq 0}, (S(t))_{t\geq 0}$  be strongly continuous linear semigroups on
a Banach space $(E,\|\cdot\|)$ consists of bounded linear operators.
Let $F\subset E$ be a dense linear subspace, equipped with a norm $\||\cdot\||$, such that both  $(T(t))_{t\geq 0}$ and $ (S(t))_{t\geq 0}$ leave $F$ invariant.
\begin{assumptionkw}\label{KWA1}
	$(T(t))_{t\geq 0}$ and $(S(t))_{t\geq 0}$ are \textbf{exponentially bounded} on $(F,\||\cdot\||)$, so there exist $M_T, M_S\geq 1$, and $\omega_T, \omega_S\in\RR$ such that\[\||T(t)\||\leq M_Te^{\omega_T t}, \quad \||S(t)\||\leq M_Se^{\omega_S t}\] for all $t\geq 0$.
\end{assumptionkw}
\begin{assumptionkw}\label{KWA2}$(T(t))_{t\geq 0}$ and $(S(t))_{t\geq 0}$ are \textbf{locally Trotter stable} on both $(E,\|\cdot\|)$ and $(F,\||\cdot\||)$. There exists $ \delta>0$ and $ M_E^\delta, M_F^\delta\geq 1$ such that 
	\[\left\|\left[T\left(\mbox{$\frac{t}{n}$}\right)S\left(\mbox{$\frac{t}{n}$}\right)\right]^n\right\|\leq M_E^\delta\]
	\[\left\|\left|\left[T\left(\mbox{$\frac{t}{n}$}\right)S\left(\mbox{$\frac{t}{n}$}\right)\right]^n\right\|\right|\leq M_F^\delta\]	for all $t\in[0,\delta]$ and $n\in\NN$.
\end{assumptionkw}
\begin{assumptionkw}\label{KWA3} (Commutator condition)
	There exists $\alpha>1$, $\delta'>0$ and $M_1\geq 0$ such that\[\left\|T(t)S(t)f-S(t)T(t)f\right\|\leq M_1t^\alpha \||f\||\] for all $f\in F$, $t\in[0,\delta]$.
\end{assumptionkw}
\begin{thrm} 	[K\"uhnemund and Wacker, \cite{Kuhnemund_Wacker}, Theorem 1]
	\label{Theorem_1_KW} 
 Let $(T(t))_{t\geq 0}$ and $(S(t))_{t\geq 0}$ be strongly continuous semigroups satisfying Assumptions $\mathrm{KW}\ref{KWA1}$-$\mathrm{KW}\ref{KWA3}$. Then the Lie-Trotter product formula holds, i.e.
 \[\mathbb{P}_tx:=\lim_{n\to\infty}\left[T\left(\mbox{$\frac{t}{n}$}\right)S\left(\mbox{$\frac{t}{n}$}\right)\right]^nx\] exists in $(E,\|\cdot\|)$ for every $x\in X$, and convergence is uniform for every $t$ in compact intervals in $\RR_+$. Moreover, $(\mathbb{P}(t))_{t\geq 0}$ is a strongly continuous semigroup in $E$.
\end{thrm}
We shall now show that Theorem \ref{Theorem_1_KW} follows from our result. Note that in Theorem \ref{Theorem_1_KW} there is no assumption that $(E,\|\cdot\|)$ should be separable, while we assume that $(S,d)$ is separable. This issue can be overcome as follows.

Fix $x\in E$. Define $T_t^1:=T(t)$, $T^2_t:=S(t)$ and \[E_x=\mathrm{Cl}_{E}\left(\mathrm{span}_\RR\left\{T_{t_N}^{i_N}\cdot T_{t_{N-1}}^{i_{N-1}}\cdot\ddots\cdot T_{t_1}^{i_1}:N\in\NN, i_k\in\{1,2\}, k=1,2,\cdots, N\right\}\right).\]
Then $E_x\subset E$ is the smallest separable closed subspace that contains $x$ and is both $(T(t))_{t\geq 0}$ and $(S(t))_{t\geq 0}$-invariant. Let $S=E_x$ with metric $d(y,y'):=\|y-y'\|$. Then $(S,d)$ is separable and complete.
\subsubsection{Lifts}
	Let $(P^1_t)_{t\geq 0}$ be the lift of $T(t)$ to $\MM^+(S)$ and $(P^2_t)_{t\geq 0}$ be the lift of $S(t)$ to $\MM^+(S)$. That is, for $\mu\in\MM^+(S)$,
	\begin{equation}
\label{lifts_KW_2}
P^1_t\mu:=\int_S \delta_{T(t)x}\mu(dx),\quad P^2_t\mu:=\int_S \delta_{S(t)x}\mu(dx),
\end{equation}
where the integrals are considered as Bochner integrals in $\overline{\MM(S)}_{\mathrm{BL}}$, the closure of $\MM(S)_{\mathrm{BL}}$ in $\BL(S,d)^*$. Since $\MM^+(S)\subset \overline{\MM(S)}_{\mathrm{BL}}$ is closed, $P_t^i\mu\in\MM^+(S)$.
So
	\begin{equation}
	\label{lifts_KW_1}
	P^1_t\delta_x:=\delta_{T(t)x}, \quad P^2_t\delta_x:=\delta_{S(t)x}.
	\end{equation}
We show that $(P^i_t)_{\geq 0}$, $i=1,2$, defined by (\ref{lifts_KW_2}) satisfy Assumptions \ref{Assumption_equicontinuity_tightness}-\ref{Assumption_4_extended_commutator}.

First consider Assumption \ref{Assumption_equicontinuity_tightness}. We discuss $(P_t^1)_{t\geq 0}$ only; the argument for $(P^2_t)_{t\geq 0}$ is similar. The map $t\mapsto P_t^1\mu:\RR_+\to\MM^+(S)_{\BL}$ is continuous iff $t\mapsto \langle P_t^1\mu,f\rangle$ is continuous for every $f\in C_b(S)$. Clearly, $\langle P_t^1\mu,f\rangle=\int_S \langle \delta_{T(t)x},f\rangle \mu(dx)=\int_S f(T(t)x)\mu (dx)$. Using the strong continuity of $(T(t))_{t\geq 0}$ and Lebesgue's Dominated Convergence Theorem we see that $t\mapsto \langle P_t^1\mu, f\rangle$ is indeed continuous on $\RR_+$. Thus, $\{P_t^1\mu:t\in [0,\delta]\}$ is compact in $\MM^+(S)_{\BL}$, that is: tight. 

	Let $\phi\in \BL(S,d)$ and $x_0\in S$. Let $U^1_t$ be dual operators to $P^1_t$. Then:
\begin{equation}
\nonumber
\begin{array}{llll}
&|U^1_t\phi(x)-U^1_t\phi(x_0)|=|\langle P^1_t\delta_x-P^1_t\delta_{x_0},\phi\rangle|\\
&\underbrace{=}_{def}|\langle \delta_{T(t)x}-\delta_{T(t)x_0},\phi\rangle|=
|\phi(\delta_{T(t)x})-\phi(\delta_{T(t)x_0})|
\underbrace{\leq}_{\phi\in \BL(S,d)} |\phi|_L\cdot\|T(t)x-T(t)x_0\|\\
&\leq |\phi|_L\cdot\|T(t)\|\cdot\|x-x_0\|\underbrace{\leq}_{KW\ref{KWA1}}|\phi|_l\cdot M_Te^{\omega_Tt}\cdot\|x-x_0\|
\end{array}
\end{equation} 
So there exists $\delta_T$ such that $\{U^1_t\phi:t\in[0,\delta_T]\}$ is equicontinuous in $C_b(S)$.
Hence, ${\{P^1_t:t\in[0,\delta_T]\}}$ forms an equicontinuous family, according to Theorem \ref{Thrm_equivalence}. 

Stability condition in Assumption \ref{Assumption_stability} can be shown as follows. 
Let $\phi\in \BL(S,d)$, $x_0\in S$.
\begin{equation}
\nonumber
\begin{array}{llll}
&\left|\left[U^2_{\frac{t}{n}}U^1_{\frac{t}{n}}\right]^n\phi(x)-\left[U^2_{\frac{t}{n}}U^1_{\frac{t}{n}}\right]^n\phi(x_0)\right|&=&\left|\left\langle\delta_x-\delta_{x_0}, \left[U^2_{\frac{t}{n}}U^1_{\frac{t}{n}}\right]^n\phi\right\rangle\right|\\
&&=&\left|\left\langle\left[P^1_{\frac{t}{n}}P^2_{\frac{t}{n}}\right]^n\delta_x-\left[P^1_{\frac{t}{n}}P^2_{\frac{t}{n}}\right]^n\delta_{x_0}, \phi\right\rangle\right|\\
&&=&\left|\left\langle\delta_{[T(\frac{t}{n})S(\frac{t}{n})]^nx}-\delta_{[T(\frac{t}{n})S(\frac{t}{n})]^nx_0}, \phi\right\rangle\right|\\
&&=&\left|\phi\left[T\left(\frac{t}{n}\right)S\left(\frac{t}{n}\right)\right]^nx-\phi\left[T\left(\frac{t}{n}\right)S\left(\frac{t}{n}\right)\right]^nx_0\right|\\
&&\leq&|\phi|_L\left\|\left[T\left(\frac{t}{n}\right)S\left(\frac{t}{n}\right)\right]^n(x-x_0)\right\|\\
&&\leq& |\phi|_L\cdot\left\|\left[T\left(\frac{t}{n}\right)S\left(\frac{t}{n}\right)\right]^n\right\|\cdot\|x-x_0\|\\
&&\leq&|\phi|_L\cdot M_E^\delta\cdot\|x-x_0\|
\end{array}
\end{equation}
by Assumption KW\ref{KWA3}, for $t\in[0,\delta]$, $n\in\NN$. Theorem \ref{Thrm_equivalence} again implies equicontinuity of $\mathcal{F}(\delta)$. 

	Let $\phi\in F\subset E$. We define 
	\[M_0:=\mathrm{span}_{\RR_+} \{\delta_\phi\,|\,\phi\in F\}\subset \mathcal{M}^+(S).\] Then $M_0$ is dense in $\MM^+(S)$ and $(P^i_t)_{t\geq 0}$-invariant, $i=1,2$.\\
	Moreover, define \begin{equation}
	\label{M0norm}
	|\mu_0|_{M_0}:=\int_F\||\phi\||\mu_0(d\phi).
	\end{equation} So
	\[\left|\sum_{k=1}^N a_k\delta_{\phi_k}\right|_{M_0}=\sum_{k=1}^N a_k\||\phi_k\||\]

To check Commutator Condition in Assumption \ref{Assumption_commutator}, let $f\in \BL(S,d)$ and $\mu_0\in M_0$. We define a new admissible metric $d_{\mathcal{E}(f)}$ as in (\ref{new_metric}). Then for $y,y'\in E_x=S$, \[d_{\mathcal{E}(f)}(y,y')=\|y-y'\|\vee \sup_{g\in\mathcal {E}(f)}|h(y)-h(y')|.\] For $h\in \mathcal{E}(f)$ there exist $s,s'$ and $t\in [0,\delta]$, with $\delta=\min(\delta_1,\delta_2)$, such that
\begin{align*}
\nonumber
|h(y)-h(y')|&=\left|f\left(\left[T\left(\mbox{$\frac{t}{n}$}\right)S\left(\mbox{$\frac{t}{n}$}\right)\right]^nT(s')S(s)y\right)-f\left(\left[T\left(\mbox{$\frac{t}{n}$}\right)S\left(\mbox{$\frac{t}{n}$}\right)\right]^nT(s')S(s)y'\right)\right|\\
&\leq |f|_{L,d}\cdot\left\|\left[T\left(\mbox{$\frac{t}{n}$}\right)S\left(\mbox{$\frac{t}{n}$}\right)\right]^nT(s')S(s)\right\|\cdot\|y-y'\|\\
&\leq M\cdot|f|_{L,d}\cdot \|y-y'\|
\end{align*}
for some constant $M>0$, according to Assumptions KW\ref{Assumption_equicontinuity_tightness}-\ref{Assumption_stability}.
	\begin{equation}
	\nonumber
	\left\|P^1_tP^2_t\mu_0-P^2_tP^1_t\mu_0\right\|^*_{\BL,d_{\mathcal{E}(f)}}\leq \int_S	\left\|P^1_tP^2_t\delta_\phi-P^2_tP^1_t\delta_\phi\right\|^*_{\BL,d_{\mathcal{E}(f)}}\mu_0(d\phi)
	\end{equation}
Let $B_{\mathcal{E}(f)}$ be the unit ball in $\BL(S,d_{\mathcal{E}(f)})$ for $\|\cdot\|_{\BL,d_{\mathcal{E}(f)}}$. 
		By the Commutator Condition KW\ref{KWA3} we get the following:
		\begin{equation}
		\nonumber
		\begin{array}{llll}
&	\left\|P^1_tP^2_t\delta_\phi-P^2_tP^1_t\delta_\phi\right\|^*_{\BL,d_{\mathcal{E}(f)}}&=&\sup_{g\in B_{\mathcal{E}(f)}}\left|g(T(t)S(t)\phi)-g(S(t)T(t)\phi)\right|\\
&&\leq&\sup_{g\in B_{\mathcal{E}(f)}}|g|_{L,d_{\mathcal{E}(f) }}\cdot d_{\mathcal{E}(f)}\left(T(t)S(t)\phi,S(t)T(t)\phi\right)\\
&&\leq& \max(1,|f|_{L,d}M)\|T(t)S(t)\phi-S(t)T(t)\phi\|\\
&&\leq& \max(1,|f|_{L,d}M)M_1t^\alpha\||\phi\||.
\end{array}
\end{equation}
Define
\[\omega_f(t,\mu_0):=\max(1,|f|_{L,d}M)M_1t^{\alpha-1}|\mu_0|_{M_0}.\]
Since $\alpha>1$,
$\omega_f:\RR_+\times M_0\to\RR_+$ is continuous, non-decreasing and for every $\delta>0$
\begin{align*}\int_0^\delta\frac{\omega_f(t,\mu_0)}{t}dt&=\max(1,|f|_{L,d}M)|\mu_0|_{M_0}M_1\int_0^\delta t^{\alpha-2}dt\\
&=\max(1,|f|_{L,d}M)M_1\frac{\delta^{\alpha -1}}{\alpha-1}<+\infty.\end{align*}
Moreover, for $\mu_0\in M_0$,
\begin{align*}
\nonumber
\left\|P^1_tP^2_t\mu_0-P^2_tP^1_t\mu_0\right\|^*_{\BL,d_{\mathcal{E}(f)}}&\leq \int_S\left\|P^1_tP^2_t\delta_\phi-P^2_tP^1_t\delta_\phi\right\|^*_{\BL,d_{\mathcal{E}(f)}}\mu_0(d\phi)\\
&\leq\max\left(1,|f|_{L,d}M\right)M_1 t^{\alpha-1}\int_S\||\phi\||\mu_0(d\phi)\\
&=t\omega_f(t,\mu_0).
\end{align*}
Hence, we get Assumption \ref{Assumption_commutator} for all $\mu_0\in  M_0$ and $\delta_{3,f}=\delta'$.

Let us now check Assumption \ref{Assumption_4_extended_commutator}. First, for any $\phi\in F$,
\begin{equation}
\nonumber
\left|\left[P^1_{\frac{t}{n}}P^2_{\frac{t}{n}}\right]^n\delta_{\phi}\right|_{M_0}=\left|\delta_{\left[T(\frac{t}{n})S(\frac{t}{n})\right]^n\phi}\right|_{M_0}=\left|\left\|\left[T\left(\mbox{$\frac{t}{n}$}\right)S\left(\mbox{$\frac{t}{n}$}\right)\right]^n\phi\right\|\right|\leq M_F^\delta\||\phi\||
\end{equation}
For $\mu_0\in M_0$ we get
\begin{equation}
\begin{array}{llll}
\nonumber
&\left|\left[P^1_{\frac{t}{n}}P^2_{\frac{t}{n}}\right]^n\mu_0\right|_{M_0}&=&\left|\left[P^1_{\frac{t}{n}}P^2_{\frac{t}{n}}\right]^n\left(\mathlarger{\sum_k}a_k\delta_{\phi_k}\right)\right|_{M_0}\\
&&=&\left|\mathlarger{\sum_k}a_k\delta_{\left[T(\frac{t}{n})S(\frac{t}{n})\right]^n\phi_k}\right|_{M_0}\\
&&=&\mathlarger{\sum_k}a_k\left|\delta_{\left[T(\frac{t}{n})S(\frac{t}{n})\right]^n\phi_k}\right|_{M_0}\\
&&\leq&\mathlarger{\sum_k}a_kM_F^\delta|\|\phi_k\||\\
&&=&M_F^\delta|\mu_0|_{M_0}.
\end{array}
\end{equation}
Furthermore, 
\begin{equation}
\nonumber
\left|P^1_t\delta_\phi\right|_{M_0}=\left|\delta_{T(t)\phi}\right|_{M_0}=\||T(t)\phi\||\leq M_Te^{\omega_Tt}\||\phi\||\leq M_Te^{\omega_T\delta}\||\phi\||
\end{equation}
and similarly
\begin{equation}
\nonumber
\left|P^2_t\delta_\phi\right|_{M_0}\leq M_Se^{\omega_S\delta}\||\phi\||.
\end{equation}
Then for $0\leq t\leq \delta$
\begin{equation}
\nonumber
\begin{array}{llll}
&\left|P^1_t\mu_0\right|_{M_0}&\leq& M_Te^{\omega_T\delta}|\mu_0|_{M_0}
\end{array}
\end{equation}
and 
\begin{equation}
\nonumber
\begin{array}{llll}
&\left|P^2_t\mu_0\right|_{M_0}&\leq& M_Se^{\omega_T\delta}|\mu_0|_{M_0}.
\end{array}
\end{equation}

Thus, 
\[\left|P^2_s\left[P^1_{\frac{t}{n}}P^2_{\frac{t}{n}}\right]^nP^1_{t'}\mu_0\right|_{M_0}\leq M_TM_SM_F^{\delta}e^{(\omega_T+\omega_s)\delta}\cdot|\mu_0|_{M_0}\]
and with $C_f(\mu_0):=M_TM_SM_F^{\delta}e^{(\omega_T+\omega_s)\delta}$ (independent of $f$ and $\mu_0$) and $\delta_{4,f}=\min(\delta,\delta')$, we see that Assumptions \ref{Assumption_equicontinuity_tightness}-\ref{Assumption_4_extended_commutator} hold. 

Hence, we conclude that the Lie-Trotter formula holds for $(P^i_t)_{t\geq 0}$, $i=1,2$. Moreover, as $\delta_{3,f}$, $\delta_{4,f}$, $C_f(\mu_0)$ and $\omega_f$ can be chosen uniformly for $f$ in the unit ball in $(\BL(S,d),\|\cdot\|_{\BL,d})$, the convergence is uniform in $f$ in compact subsets of $\RR_+$. Furthermore, for every $y\in E_x$,  $\left[P^1_{\frac{t}{n}}P^2_{\frac{t}{n}}\right]^n\delta_y=\delta_{\left[T\left(\frac{t}{n}\right)S\left(\frac{t}{n}\right)\right]^ny}\to \overline{\mathbb{P}}_t\delta_y$ in $\MM^+(S)_{\BL}$ as $n\to\infty$.

	The set of Dirac measures is closed in $\MM^+(S)_{\BL}$. To show this let $(\delta_{x_n})_n$ be a sequence of Dirac measures such that $\delta_{x_n}\to \mu$ for some $\mu\in\MM^+(S)$. Then $(\delta_{x_n})_n$ is a Cauchy sequence, and
\[\|\delta_{x_n}-\delta_{x_m}\|_{BL,d}^*=\frac{2d(x_n,x_m)}{2+d(x_n,x_m)}\]
(cf. \cite{Hille2009_embedding} Lemma 2.5). Then also $(x_n)_{n\in\NN}\subset S$ is a Cauchy sequence. As $S$ is complete, $(x_n)_{n\in\NN}$ is convergent. Hence, there exists $x^*\in S$ such that $x_n\to x^*$ as $n\to \infty$ and
\[\|\delta_{x_n}-\delta_{x^*}\|_{BL,d}^*=\frac{2d(x_n,x^*)}{2+d(x_n,x^*)}\to 0\text{ as } n\to \infty.\]
Hence,  $\overline{\mathbb{P}}_t\delta_y=\delta_{\mathbb{P}_t^xy}$ for a specific $\mathbb{P}_t^x\subset E$ (as in statement Theorem \ref{Theorem_1_KW}). Because the $(P^i_t)_{t\geq 0}$, $i=1,2$,  are strongly continuous in this setting, $(\overline{\mathbb{P}}_t)_{t\geq 0}$ is a semigroup by Proposition \ref{Proposition_semigroup}. Therefore, $(\mathbb{P}_t^x)_{t\geq 0}$ is a strongly continuous semigroup on $E_x$. 
The operators $\mathbb{P}_t$ are linear and continuous:

Let $y_n\in E_x$ such that $Y_n\to y$ in $E$. Then
\begin{align*}
\|\mathbb{P}_t^xy_n-\mathbb{P}_t^xy\|_{\BL,d}^*&=\frac{2\|\delta_{\mathbb{P}_ty_n}-\delta_{\mathbb{P}_ty}\|^*_{\BL,d}}{2+\|\delta_{\mathbb{P}_ty_n}-\delta_{\mathbb{P}_ty}\|^*_{\BL,d}}\\
&=\frac{2\|\overline{\mathbb{P}}_t\delta_{y_n}-\overline{\mathbb{P}}_t\delta_{y}\|_{\BL,d}^*}{2+\|\overline{\mathbb{P}}_t\delta_{y_n}-\overline{\mathbb{P}}_t\delta_{y}\|^*_{\BL,d}}\to 0.
\end{align*}
Moreover, $E=\bigcup _{x\in E}E_x$, and the semigroups $(\mathbb{P}^x_t)_{t\geq 0}$ and $(\mathbb{P}^{x'}_t)_{t\geq 0}$
agree on $E_x\cap E_{x'}$. This allows us to define a strongly continuous semigroup $(\mathbb{P}_t)_{t\geq 0}$ of bounded linear operators on $E$ that agrees with $(\mathbb{P}^x_t)_{t\geq 0}$ on $E_x$. 
\subsection{Colombo-Guerra}

Colombo and Guerra in \cite{Colombo_2009}, generalizing Colombo and Corli \cite{Colombo2004},  also established conditions that ensure the convergence of the Lie-Trotter formula for linear semigroups in a Banach space that do not involve the domains of their generators. Instead, like in the results of K\"uhnemund and Wacker \cite{Kuhnemund_Wacker}, they build on a commutator condition (Assumption CG\ref{CC2commutator} stated below) that is weaker than that in \cite{Kuhnemund_Wacker}. It is this condition that motivated our Assumption \ref{Assumption_commutator}.

The situation in  \cite{Colombo_2009} is as follows. Let $S^1, S^2:\RR_+\times X\mapsto X$ be strongly continuous semigroups on a Banach space $X$. Assume that there exists a normed vector space $Y$ which is densely embedded in $X$ and invariant under both semigroups such that:
\begin{assumptionCC2}
	\label{Assumtpion_CG_1}
	The two semigroups are \textbf{locally Lipschitz in time in $Y$}, i.e. there exists a compact map $K:Y\to\RR$ such that for $i=1,2$
	\[\left\|S_t^1u-S_{t'}^iu\right\|_X\leq K(u)|t-t'|\quad \text{ for all } u\in Y,\,\, t,t'\in I.\]
\end{assumptionCC2}
\begin{assumptionCC2}
		\label{Assumtpion_CG_2}
	The two semigroups are \textbf{exponentially bounded on $F$} and \textbf{locally Trotter stable on $X$ and $Y$}, i.e. there exists a constant $H$ such that for all $t\in [0,1]$, $n\in\NN$
	\[\|S_t^1\|_Y+\|S_t^2\|_Y+\left\|\left(S_{\frac{t}{n}}^1S_{\frac{t}{n}}^2\right)^n\right\|_X+\left\|\left(S_{\frac{t}{n}}^1S_{\frac{t}{n}}^2\right)^n\right\|_Y\leq H.\]
\end{assumptionCC2}
\begin{assumptionCC2}[Commutator condition]
	\label{CC2commutator}
\[\left\|S_{t}^1S_{t}^2u-S_t^2S_t^1u\right\|_X\leq t\omega(t)\|u\|_Y\]
is satisfied for all $u\in Y$ and $t\in[0,\delta]$ with some $\delta>0$, and for a suitable $\omega:[0,\delta]\to \RR^+$ with $\int_0^\delta \frac{\omega(\tau)}{\tau}d\tau<+\infty$.
\end{assumptionCC2}
\begin{thrm}
	\label{Thrm_CG}
Under Assumptions CG\ref{Assumtpion_CG_1}-CG\ref{CC2commutator}  there exists a global semigroup $Q:[0,+\infty)\times X\to X$ such that for all $u\in Y$, there exists a constant $C_u$ such that for $t>0$
\[\frac{1}{t}\left\|Q(t)u-S_t^1S_t^2u\right\|_X\leq C_u\int_0^t\frac{\omega(\xi)}{\xi}d\xi.\]
\end{thrm}

In fact,  \cite{Colombo_2009} Proposition 3.5 also includes a statement of convergence of so-called Euler polygonals to orbits of $Q$. The interested reader should consult \cite{Colombo_2009} for further details on this topic.

The construction in this case that allows us to conclude Theorem \ref{Thrm_CG} from our Theorem \ref{main} is highly similar to the K\"uhnemund-Wacker case discussed in the previous section. Therefore we state the main reasoning and give the immediate results. 

Let $u\in X$. We take $S=X_u$ where the latter is the smallest separable Banach space in $X$ that is invariant under $(S_t^i)_{t\geq 0}$, $i=1,2$, equipped with the metric induced by the norm on $X$.  Let $P^1_t$ and $P^2_t$ be lifts of $S^1_t$ and $S^2_t$ to $\MM^+(S)$:
\[P^i_t\delta_u:=\delta_{S_t^iu},\, P_t^i\mu:=\int_{\mathcal{U}}\delta_{S_t^iu}\mu(du), i=1,2.\]
Now we check if $P^1_t$ and $P^2_t$ satisfy Assumptions \ref{Assumption_equicontinuity_tightness}-\ref{Assumption_4_extended_commutator}.\\
As in Section \ref{KW_case}, because $(S^1_t)_{t\geq 0}$ and $(S^2_t)_{t\geq 0}$ are strongly continuous semigroups,  $(P^1_t)_{t\geq 0}$ and $(P^2_t)_{t\geq 0}$ are tight. Moreover, if $\phi\in \BL(S,d)$ and $v,w\in X_u$, and $U^1_t$ and $U^2_t$ are the dual operators of $P^1_t$ and $P_t^2$ respectively, then:
\begin{equation}\nonumber
\begin{array}{llllll}
\left|U^1_t\phi(v)-U^1_t\phi(w)\right|\leq|\phi|_L\cdot H\cdot\|v-w\|_X
\end{array}
\end{equation}
This yields the equicontinuity condition for $U^1_t$. Similarly equicontinuity for $U_t^2$ is established.  \\
A similar computation yields Assumption \ref{Assumption_stability}:
\begin{align*}
\nonumber
\left| \left[U^1_{\frac{t}{n}}U^2_{\frac{t}{n}}\right]^n\phi(v)-\left[U^1_{\frac{t}{n}}U^2_{\frac{t}{n}}\right]^n\phi(w)\right|&=\left|\phi\left[\left(S^2_{\frac{t}{n}}S^1_{\frac{t}{n}}\right)^nv\right]-\phi\left[\left(S^2_{\frac{t}{n}}S^1_{\frac{t}{n}}\right)^nw\right]\right|\\
&\leq |\phi|_L\cdot \left\|\left(S^2_{\frac{t}{n}}S^1_{\frac{t}{n}}\right)^n(v-w)\right\|_X
\leq |\phi|_L\cdot H \cdot \|v-w\|_X\\
\end{align*}
To check the  Commutator Condition in Assumption \ref{Assumption_commutator}, let $f\in \BL(S,d)$, put $M_0:=\mathrm{span}\{\delta_v|v\in Y\cap X_u\}$ and $|\mu_0|_{M_0}$ as in (\ref{M0norm}). Then define 
\[\omega_f(t,\mu_0):=\max(1,|f|_{L,d}M)\omega(t)|\mu_0|_{M_0}.\]
Commutator Condition CG\ref{CC2commutator} yields
\begin{align*}
\nonumber
\left\|P^1_tP^2_t\delta_u-P^2_tP^1_t\delta_u\right\|^*_{\BL,d_{\mathcal{E}(f)}}\leq \max(1,|f|_{L,d}M)t\omega(t)\|u\|_Y
\end{align*}
as before, which established Assumption \ref{Assumption_commutator}. Note that $\omega_f$ can be chosen uniformly for $f$ in the unit ball of $\BL(S,d)$.
 
 Assumption \ref{Assumption_4_extended_commutator} is obtained from the estimate
\begin{equation}
\nonumber
\left|\left[P^1_{\frac{t}{n}}P^2_{\frac{t}{n}}\right]^n\delta_{u}\right|_{M_0}=\left|\delta_{\left[S_{\frac{t}{n}}^1S^2_{\frac{t}{n}}\right]^n u}\right|_{M_0}=\left\|\left[S_{\frac{t}{n}}^2S^1_{\frac{t}{n}}\right]^nu\right\|_X\leq H\|u\|_X,
\end{equation}
which yields
\begin{align*}
\nonumber
\left|\left[P^1_{\frac{t}{n}}P^2_{\frac{t}{n}}\right]^n\mu_0\right|_{M_0}
\leq H|\mu_0|_{M_0}.
\end{align*}
and
\begin{equation}
\nonumber
\left|P^1_t\delta_\phi\right|_{M_0}=\left|\delta_{S_t^1u}\right|_{M_0}=\|S^1_tu\|_Y\leq H\|u\|_Y, \quad \left|P^2_t\delta_\phi\right|_{M_0}\leq H\|u\|_Y
\end{equation}
which yields
\begin{align*}
\nonumber
\left|P^1_t\mu_0\right|_{M_0}\leq H|\mu_0|_{M_0}\text{ and } \left|P^2_t\mu_0\right|_{M_0}\leq H|\mu_0|_{M_0}.
\end{align*}
Thus, the Lie-Trotter formula holds for $(P^1_t)_{t\geq 0}$ and $(P^2_t)_{t\geq 0}$. A similar argument as in Section \ref{KW_case} yields Theorem \ref{Thrm_CG}.
\appendix
\section*{Appendices}
\addcontentsline{toc}{section}{Appendices}
\renewcommand{\thesubsection}{\Alph{subsection}}

\subsection{Proof of Lemma \ref{lem_est}}
\label{appendix_proofs}
\label{App_2_Proof_of_Lemma_lem_est}
	\label{appendix_1}

		(a) 	We will check it by induction on $j$. Let $j=1$. Then the left hand side in the equation \ref{lem_est}, (a) is of the form
		\begin{equation}
		\nonumber
		\begin{array}{llll}
		L=	P^1_{\frac{t}{m}}P^2_\frac{t}{m}-P^2_{\frac{t}{m}}P^1_{\frac{t}{m}},
		\end{array}
		\end{equation}
		while the right hand side is 
		\begin{equation}
		\nonumber
		\begin{array}{llll}
		R=\sum_{l=0}^{0}P^2_{\frac{lt}{m}}\left(P^1_{\frac{t}{m}}P^2_{\frac{t}{m}}-P^2_{\frac{t}{m}}P^1_{\frac{t}{m}}\right)P^2_{\frac{(1-1-l)t}{m}}=P^2_{\frac{0t}{m}}\left(P^1_{\frac{t}{m}}P^2_{\frac{t}{m}}-P^2_{\frac{t}{m}}P^1_{\frac{t}{m}}\right)P^2_{\frac{(1-1-0)t}{m}}=L
		\end{array}
		\end{equation}
		Assume that (a) holds for $j-1$:
		\begin{equation}
		\nonumber
		\begin{array}{llll}
		P^1_{\frac{t}{m}}P^2_\frac{(j-1)t}{m}-P^2_{\frac{(j-1)t}{m}}P^1_{\frac{t}{m}}=\sum_{l=0}^{j-2}P^2_{\frac{lt}{m}}\left(P^1_{\frac{t}{m}}P^2_{\frac{t}{m}}-P^2_{\frac{t}{m}}P^1_{\frac{t}{m}}\right)P^2_{\frac{(j-2-l)t}{m}}
		\end{array}
		\end{equation}
		Then for $j$:
		\begin{equation}
		\nonumber
		\begin{array}{llll}
		&L&=&P^1_{\frac{t}{m}}P^2_\frac{jt}{m}-P^2_{\frac{jt}{m}}P^1_{\frac{t}{m}}\\
		&&=&\left(P^1_{\frac{t}{m}}P^2_\frac{(j-1)t}{m}-P^2_{\frac{(j-1)t}{m}}P^1_{\frac{t}{m}} \right)P^2_{\frac{t}{m}}+P^2_{\frac{(j-1)t}{m}}P^1_{\frac{t}{m}}P^2_{\frac{t}{m}}-P^2_{\frac{jt}{m}}P^1_{\frac{t}{m}}\\
		&&=&\left(\sum_{l=0}^{j-2}P^2_{\frac{lt}{m}}\left(P^1_{\frac{t}{m}}P^2_{\frac{t}{m}}-P^2_{\frac{t}{m}}P^1_{\frac{t}{m}}\right)P^2_{\frac{(j-2-l)t}{m}} \right)P^2_{\frac{t}{m}}+P^2_{\frac{(j-1)t}{m}}\left(P^1_{\frac{t}{m}}P^2_{\frac{t}{m}}-P^2_{\frac{t}{m}}P^1_{\frac{t}{m}}\right)\\
		&&=&\sum_{l=0}^{j-2}P^2_{\frac{lt}{m}}\left(P^1_{\frac{t}{m}}P^2_{\frac{t}{m}}-P^2_{\frac{t}{m}}P^1_{\frac{t}{m}}\right)P^2_{\frac{(j-1-l)t}{m}} +P^2_{\frac{(j-1)t}{m}}\left(P^1_{\frac{t}{m}}P^2_{\frac{t}{m}}-P^2_{\frac{t}{m}}P^1_{\frac{t}{m}}\right)
		\\
		&&=&\sum_{l=0}^{j-1}P^2_{\frac{lt}{m}}\left(P^1_{\frac{t}{m}}P^2_{\frac{t}{m}}-P^2_{\frac{t}{m}}P^1_{\frac{t}{m}}\right)P^2_{\frac{(j-1-l)t}{m}} 		=R
		\end{array}
		\end{equation}
		(b) 	We will check it by induction on $k$. Let $k=2$.
		\begin{equation}
		\nonumber
		\begin{array}{llll}
		L=P^1_{\frac{2t}{m}}P^2_\frac{2t}{m}-\left(P^1_{\frac{t}{m}}P^2_{\frac{t}{m}}\right)^2
		\end{array}
		\end{equation}
		\begin{equation}
		\nonumber
		\begin{array}{llll}
		&			R&=&\sum_{j=1}^{1}P^1_{\frac{tj}{m}}\left(P^1_{\frac{t}{m}}P^2_\frac{jt}{m}-P^2_{\frac{jt}{m}}P^1_{\frac{t}{m}}\right)P^2_{\frac{t}{m}}\left(P^1_{\frac{t}{m}}P^2_{\frac{t}{m}}\right)^{2-1-j}\\
		&&=&P^1_{\frac{t}{m}}\left(P^1_{\frac{t}{m}}P^2_\frac{t}{m}-P^2_{\frac{t}{m}}P^1_{\frac{t}{m}}\right)P^2_{\frac{t}{m}}=L
		\end{array}
		\end{equation}
		Assume that for $k-1$ we have:
		\begin{equation}
		\nonumber
		\begin{array}{llll}
		P^1_{\frac{(k-1)t}{m}}P^2_\frac{(k-1)t}{m}-\left(P^1_{\frac{t}{m}}P^2_{\frac{t}{m}}\right)^{k-1}=\sum_{j=1}^{k-2}P^1_{\frac{tj}{m}}\left(P^1_{\frac{t}{m}}P^2_\frac{jt}{m}-P^2_{\frac{jt}{m}}P^1_{\frac{t}{m}}\right)P^2_{\frac{t}{m}}\left(P^1_{\frac{t}{m}}P^2_{\frac{t}{m}}\right)^{k-2-j}
		\end{array}
		\end{equation}
		Then for $k$ we have:
		\begin{equation}
		\nonumber
		\begin{array}{llll}
		&L&=&P^1_{\frac{kt}{m}}P^2_\frac{kt}{m}-\left(P^1_{\frac{t}{m}}P^2_{\frac{t}{m}}\right)^k\\
		&&=&\left[P^1_{\frac{(k-1)t}{m}}P^2_{\frac{(k-1)t}{m}}-\left(\PP\PQ\right)^{k-1}\right]\PP\PQ-P^1_{\frac{(k-1)t}{m}}P^2_{\frac{(k-1)t}{m}}\PP\PQ+P^1_{\frac{kt}{m}}P^2_{\frac{kt}{m}}\\
		&&=&\left[\sum_{j=1}^{k-2}P^1_{\frac{tj}{m}}\left(P^1_{\frac{t}{m}}P^2_\frac{jt}{m}-P^2_{\frac{jt}{m}}P^1_{\frac{t}{m}}\right)P^2_{\frac{t}{m}}\left(P^1_{\frac{t}{m}}P^2_{\frac{t}{m}}\right)^{k-2-j}\right]\PP\PQ\\
		&&&-P^1_{\frac{(k-1)t}{m}}\left(P^2_{\frac{(k-1)t}{m}}\PP-P^1_{\frac{t}{m}}P^2_{\frac{(k-1)t}{m}}\right)\PQ\\
		&&=&\sum_{j=1}^{k-1}P^1_{\frac{tj}{m}}\left(P^1_{\frac{t}{m}}P^2_\frac{jt}{m}-P^2_{\frac{jt}{m}}P^1_{\frac{t}{m}}\right)P^2_{\frac{t}{m}}\left(P^1_{\frac{t}{m}}P^2_{\frac{t}{m}}\right)^{k-1-j}\\
		&&=&R
		\end{array}
		\end{equation}
		(c) Let $n=1$. Then
		\begin{equation}
		\nonumber
		\begin{array}{llll}
		L=	P^1_{\frac{kt}{m}}P^2_\frac{kt}{m}-\left(P^1_{\frac{t}{m}}P^2_{\frac{t}{m}}\right)^{ k}
		\end{array}
		\end{equation}
		\begin{equation}
		\nonumber
		\begin{array}{llll}
		R=\left(P^1_{\frac{kt}{m}}P^2_{\frac{kt}{m}}\right)^0\left[P^1_{\frac{kt}{m}}P^2_{\frac{kt}{m}}-\left(\PP\PQ\right)^k\right]\left(\PP\PQ\right)^{k\cdot(1-1-0)}=L
		\end{array}
		\end{equation}
		Now let's assume that
		\begin{equation}
		\nonumber
		\begin{array}{llll}
		\left(P^1_{\frac{kt}{m}}P^2_\frac{kt}{m}\right)^{n-1}-\left(P^1_{\frac{t}{m}}P^2_{\frac{t}{m}}\right)^{(n-1)\cdot k}=\\=\sum_{i=0}^{n-2}\left(P^1_{\frac{kt}{m}}P^2_{\frac{kt}{m}}\right)^i\left[P^1_{\frac{kt}{m}}P^2_{\frac{kt}{m}}-\left(\PP\PQ\right)^k\right]\left(\PP\PQ\right)^{k\cdot(n-2-i)}
		\end{array}
		\end{equation}
		and let us check for $n$:
		\begin{equation}
		\nonumber
		\begin{array}{llll}
		L=\left(P^1_{\frac{kt}{m}}P^2_\frac{kt}{m}\right)^{n}-\left(P^1_{\frac{t}{m}}P^2_{\frac{t}{m}}\right)^{n\cdot k}\\
		=\left(\left(P^1_{\frac{kt}{m}}P^2_\frac{kt}{m}\right)^{n-1}-\left(P^1_{\frac{t}{m}}P^2_{\frac{t}{m}}\right)^{(n-1)\cdot k}\right)\left(\PP\PQ\right)^k-\left(P^1_{\frac{kt}{m}}P^2_\frac{kt}{m}\right)^{n-1}\left(\PP\PQ\right)^k\\
		+\left(P^1_{\frac{kt}{m}}P^2_\frac{kt}{m}\right)^{n}\\
		=\left(\sum_{i=0}^{n-2}\left(P^1_{\frac{kt}{m}}P^2_{\frac{kt}{m}}\right)^i\left(P^1_{\frac{kt}{m}}P^2_{\frac{kt}{m}}-\left(\PP\PQ\right)^k\right)\left(\PP\PQ\right)^{k\cdot(n-2-i)}\right)\left(\PP\PQ\right)^k-\\
		-\left(P^1_{\frac{kt}{m}}P^2_\frac{kt}{m}\right)^{n-1}\left(\PP\PQ\right)^k+\left(P^1_{\frac{kt}{m}}P^2_\frac{kt}{m}\right)^{n}\\
		=\sum_{i=0}^{n-2}\left(P^1_{\frac{kt}{m}}P^2_{\frac{kt}{m}}\right)^i\left(P^1_{\frac{kt}{m}}P^2_{\frac{kt}{m}}-\left(\PP\PQ\right)^k\right)\left(\PP\PQ\right)^{k\cdot(n-1-i)}+\\
		+\left(P^1_{\frac{kt}{m}}P^2_\frac{kt}{m}\right)^{n-1}\left[P^1_{\frac{kt}{m}}P^2_\frac{kt}{m}-\left(\PP\PQ\right)^k\right]=\\
		=\sum_{i=0}^{n-1}\left(P^1_{\frac{kt}{m}}P^2_{\frac{kt}{m}}\right)^i\left[P^1_{\frac{kt}{m}}P^2_{\frac{kt}{m}}-\left(\PP\PQ\right)^k\right]\left(\PP\PQ\right)^{k\cdot(n-1-i)}=R\quad\quad\qed
		\end{array}
		\end{equation}	
\subsection{Proof of Lemma \ref{lem_2 new}}
	\label{App_3_Proof_of_Lemma_lem_2 new}
	 Let $n\in\NN$, $k\in\NN$ and $m:=kn$ be such that $\frac{t}{nk}\in[0,\delta_f]$. Then by Lemma \ref{lem_est} (c) we get 
		\begin{equation}
		\nonumber
		\begin{array}{llll}
		\left|\left\langle \left[P^1_{\frac{kt}{m}}P^2_\frac{kt}{m}\right]^n\mu_0-\left[P^1_{\frac{t}{m}}P^2_{\frac{t}{m}}\right]^{n\cdot k}\mu_0,f\right\rangle\right|\\
		=\left|\left\langle \sum_{i=0}^{n-1}\left[P^1_{\frac{kt}{m}}P^2_{\frac{kt}{m}}\right]^i\left(P^1_{\frac{kt}{m}}P^2_{\frac{kt}{m}}-\left[\PP\PQ\right]^k\right)\left[\PP\PQ\right]^{k\cdot(n-1-i)}\mu,f\right\rangle\right|\\
		\leq\sum_{i=0}^{n-1}\left|\left\langle \left[P^1_{\frac{kt}{m}}P^2_{\frac{kt}{m}}\right]^i\left(P^1_{\frac{kt}{m}}P^2_{\frac{kt}{m}}-\left[\PP\PQ\right]^k\right)\left[\PP\PQ\right]^{k\cdot(n-1-i)}\mu,f\right\rangle\right|=(**)\\
		\end{array}
		\end{equation}
		by Lemma \ref{lem_est} (b)
		\begin{equation}
		\nonumber
		\begin{array}{llll}
		(**)=\sum_{i=0}^{n-1}\left|\left\langle \left[P^1_{\frac{kt}{m}}P^2_{\frac{kt}{m}}\right]^i\left(\sum_{j=1}^{k-1}P^1_{\frac{tj}{m}}\left(P^1_{\frac{t}{m}}P^2_\frac{jt}{m}-P^2_{\frac{jt}{m}}P^1_{\frac{t}{m}}\right)P^2_{\frac{t}{m}}\times\right.\right.\right.\\
		\left.\left.\left.\times\left[P^1_{\frac{t}{m}}P^2_{\frac{t}{m}}\right]^{k-1-j}\right)\left[\PP\PQ\right]^{k\cdot(n-1-i)}\mu_0,f\right\rangle\right|\\
		\leq \sum_{i=0}^{n-1}\sum_{j=1}^{k-1}\left|\left\langle \left[P^1_{\frac{kt}{m}}P^2_{\frac{kt}{m}}\right]^iP^1_{\frac{tj}{m}}\left(P^1_{\frac{t}{m}}P^2_\frac{jt}{m}-P^2_{\frac{jt}{m}}P^1_{\frac{t}{m}}\right)P^2_{\frac{t}{m}}\times\right.\right.\\
		\left.\left.\times\left[P^1_{\frac{t}{m}}P^2_{\frac{t}{m}}\right]^{k(n-i)-1-j}\mu_0,f\right\rangle\right|=(***)\\
		\end{array}
		\end{equation}
		by Lemma \ref{lem_est} (a) we get
		\begin{equation}
		\nonumber
		\begin{array}{llll}
		(***)=\sum_{i=0}^{n-1}\sum_{j=1}^{k-1}\left|\left\langle \left[P^1_{\frac{kt}{m}}P^2_{\frac{kt}{m}}\right]^iP^1_{\frac{tj}{m}}\left(\sum_{l=0}^{j-1}P^2_{\frac{lt}{m}}\left(P^1_{\frac{t}{m}}P^2_{\frac{t}{m}}-P^2_{\frac{t}{m}}P^1_{\frac{t}{m}}\right)P^2_{\frac{(j-1-l)t}{m}}\right)P^2_{\frac{t}{m}}\times\right.\right.\\
		\left.\left.\times\left[P^1_{\frac{t}{m}}P^2_{\frac{t}{m}}\right]^{k(n-i)-1-j}\mu_0,f\right\rangle\right|\\
		\leq \sum_{i=0}^{n-1}\sum_{j=1}^{k-1}\sum_{l=0}^{j-1}\left|\left\langle \left[P^1_{\frac{kt}{m}}P^2_{\frac{kt}{m}}\right]^iP^1_{\frac{tj}{m}}\left(P^2_{\frac{lt}{m}}\left(P^1_{\frac{t}{m}}P^2_{\frac{t}{m}}-P^2_{\frac{t}{m}}P^1_{\frac{t}{m}}\right)P^2_{\frac{(j-1-l)t}{m}}\right)P^2_{\frac{t}{m}}\times\right.\right.\\
		\left.\left.\times\left[P^1_{\frac{t}{m}}P^2_{\frac{t}{m}}\right]^{k(n-i)-1-j}\mu_0,f\right\rangle\right|\\
		= \sum_{i=0}^{n-1}\sum_{j=1}^{k-1}\sum_{l=0}^{j-1}\left|\left\langle \left(P^1_{\frac{t}{m}}P^2_{\frac{t}{m}}-P^2_{\frac{t}{m}}P^1_{\frac{t}{m}}\right)P^2_{\frac{(j-l)t}{m}}\left[P^1_{\frac{t}{m}}P^2_{\frac{t}{m}}\right]^{k(n-i)-1-j}\mu_0,\right.\right.\\
		\left.\left.U^2_{\frac{lt}{m}}U^1_{\frac{tj}{m}}\left[U^2_{\frac{kt}{m}}U^1_{\frac{kt}{m}}\right]^if\right\rangle\right|\\
		\end{array}
		\end{equation}
		For every $i,j,l\in\NN$ we get
		\[g^n_{i,j,l}:=U^2_{\frac{lt}{m}}U^1_{\frac{tj}{m}}\left[U^2_{\frac{kt}{m}}U^1_{\frac{kt}{m}}\right]^if\in \mathcal{E}(f).\]
		Let $\nu^n_{i,j,l}:=P^2_{\frac{(j-l)t}{m}}\left[P^1_{\frac{t}{m}}P^2_{\frac{t}{m}}\right]^{k(n-i)-1-j}\mu$. Then $\nu^n_{i,j,l}\in M_0$. Note that $\left\|g^n_{i,j,l}\right\|_{\BL,d_{\mathcal{E}(f)}}\leq 1$\\

		Using Assumption \ref{Assumption_4_extended_commutator} we get:
		\begin{equation}
		\nonumber
		\begin{array}{llll}
		\sum_{i=0}^{n-1}\sum_{j=1}^{k-1}\sum_{l=0}^{j-1}\left|\left\langle \left(P^1_{\frac{t}{m}}P^2_{\frac{t}{m}}-P^2_{\frac{t}{m}}P^1_{\frac{t}{m}}\right)\nu^n_{i,j,l},g^n_{i,j,l}\right\rangle\right|\\
		\leq \sum_{i=0}^{n-1}\sum_{j=1}^{k-1}\sum_{l=0}^{j-1}\left|\left\langle \left(P^1_{\frac{t}{m}}P^2_{\frac{t}{m}}-P^2_{\frac{t}{m}}P^1_{\frac{t}{m}}\right)\nu^n_{i,j,l},g^n_{i,j,l}\right\rangle\right|\\
		\leq \sum_{i=0}^{n-1}\sum_{j=1}^{k-1}\sum_{l=0}^{j-1}\left\|\left(P^1_{\frac{t}{m}}P^2_{\frac{t}{m}}-P^2_{\frac{t}{m}}P^1_{\frac{t}{m}}\right)\nu^n_{i,j,l}\right\|^*_{\BL,d_\mathcal{E(f)}}\cdot\left\|g^n_{i,j,l}\right\|_{\BL,d_\mathcal{E(f)}}\\
		\leq \sum_{i=0}^{n-1}\sum_{j=1}^{k-1}\sum_{l=0}^{j-1}\frac{t}{m}\omega_f\left(\frac{t}{m},P^2_{\frac{(j-l)t}{m}}\left[P^1_{\frac{t}{m}}P^2_{\frac{t}{m}}\right]^{k(n-i)-1-j}\mu_0\right)\\
		\leq \frac{t}{m}\sum_{i=0}^{n-1}\sum_{j=1}^{k-1}\sum_{l=0}^{j-1}C_2(\mu_0)\omega_f\left(\frac{t}{m},\mu_0\right)\\
		\leq C_f(\mu_0) \frac{t}{m}\omega_f\left(\frac{t}{m},\mu_0\right)\sum_{i=0}^{n-1}\sum_{j=1}^{k-1}\sum_{l=0}^{j-1}1\\
		=C_f(\mu_0)\frac{t}{m}\omega_f\left(\frac{t}{m},\mu_0\right)\frac{n(k-1)k}{2}
		\end{array}
		\end{equation}
		So with $m=nk$ we get the result. 	\qed

\end{document}